%% file: main.tex
\documentclass[oneside]{amsart}
\usepackage[color=blue!20!white,textsize=tiny]{todonotes}

\usepackage{bm}
\usepackage{multicol}

\usepackage{mathrsfs}
\usepackage{amsmath}
\usepackage{amsthm}
\usepackage{amssymb} 
\usepackage{amsfonts}
\usepackage{tikz-cd}
\usepackage{graphicx}
\usepackage{xcolor}
\definecolor{darkgreen}{rgb}{0,0.5,0}
\usepackage[normalem]{ulem}
\usepackage{comment}

\usepackage{caption,subcaption,pinlabel}
\usepackage{graphics}
\usepackage[pagebackref=true]{hyperref}
\usepackage{scalerel}
\usepackage{comment}
\usepackage{mathtools}
\usepackage{tikz}
\usetikzlibrary{matrix,arrows,decorations.pathmorphing}
\usepackage{mathabx}
\usetikzlibrary{matrix}
\usepackage{scrextend}

\input{preamble.tex}

\usepackage{quiver}
\usepackage{comment}
\usepackage{tikz-cd}
\usepackage{xcolor}
\title{Exotic knot traces and the epsilon-invariant}
\author[S. Sinha]{Shreya Sinha}
\address {Department of Mathematics, Dartmouth College, Hanover, NH 03755}
\email{shreya.sinha.gr@dartmouth.edu}

\author[A. Tsypin]{Allison Tsypin}
\address {Department of Mathematics, Dartmouth College, Hanover, NH 03755}
\email{allison.l.tsypin.gr@dartmouth.edu}

\begin{document}

\begin{abstract}
Building on recent work of Hom and Wan, we construct additional examples of exotic knot traces, providing partial progress toward their conjecture. Our approach uses the immersed-curve techniques in Heegaard Floer homology developed by Chen and Hanselman.
\end{abstract}

\maketitle

\section{Introduction}
The $n$-trace of a knot $K$, denoted $X_n(K)$, is the smooth manifold obtained by attaching an $n$-framed $2$-handle along $K$ to the boundary of  $D^4$. Exotic knot traces have been extensively studied; the first exotic knot traces were introduced by Akbulut \cite{akbulut-1}, \cite{akbulut-2} using Donaldson invariants (see also \cite{akbulut-matveyev-1997}). This example was later generalized to an infinite family of exotic pairs by Kalmár and Stipsicz \cite{kalmar-stipsicz-2013} using Seiberg-Witten theory. Further developments came from Yasui \cite{yasui} who constructed the two families of satellite knots with patterns $P_{n, m}$ and $Q_{n, m}$ (see Figure~\ref{fig: pnm qnm}) and showed that for certain $n,m$ and conditions on the companion knot $K$, the $n$-traces $X_n(P_{n,m}(K))$ and $X_n(Q_{n,m}(K))$ are exotic \cite[Theorem 4.1]{yasui}. Yasui additionally showed that that for any integers $n,m$ and any knot $K$, the knot traces $X_n(P_{n,m}(K))$ and $X_n(Q_{n,m}(K))$ are homeomorphic (\cite[Lemma 4.3]{yasui}).  

One approach to finding exotic knot traces involves upgrading knot concordance invariants to those on knot traces. Three notable knot concordance invariants examined include the $\tau$-invariant introduced by Ozsváth and Szabó \cite{ozsvath2003knot} and independently by Rasmussen \cite{rasmussen2003floer}, the $\varepsilon$-invariant defined by Hom \cite{hom2014bordered}, and the $\nu$-invariant defined by Ozsváth and Szabó in \cite{ozsvath2010knot}. In particular, Hayden, Mark, and Piccirillo \cite{hayden-mark-piccirillo-2021} upgraded the concordance invariant $\nu$ to an $n$-trace invariant for almost all $n$ and knots $K$, and used it to find additional exotic pairs (also using the pattern knots $P_{n,m}, Q_{n,m}$). They additionally showed that that $\tau$ and $\varepsilon$ were not $n$-trace invariants; implicit in their proof (as pointed out in \cite{homwan2026}) is that the absolute value of $\varepsilon$ is a $0$-trace invariant. More recently, using the $s$-invariant, Ren and Willis \cite{ren2024khovanov} also found exotic trace pairs involving the pattern knots $P_{n,m}, Q_{n,m}$ for certain conditions on $K$ and $m \geq 0$.  
    
In \cite{homwan2026}, Hom and Wan show that $\nu$ is an $n$-trace invariant for all cases $n$ and all knots $K$, finishing the result of \cite{hayden-mark-piccirillo-2021}. They also show the absolute value of the $\varepsilon$-invariant from knot Floer homology is an $n$-trace invariant. In particular, Hom and Wan use these results to obstruct diffeomorphisms between $X_n(P_{n, m}(K))$ and $X_n(Q_{n, m}(K))$ with certain conditions on $n, m$ and $\varepsilon(K)$. They show the following three cases using the $\nu$ and $\abs{\varepsilon}$ knot trace invariants. 

\newpage
    
\begin{theorem}\label{hom-wan-cases}(\cite[Theorem 1.3]{homwan2026}) Let $K$ be a knot in $S^3$ and $m$ any non-negative integer.
\begin{enumerate}
    \item If $\varepsilon(K)=0$, then $X_n(P_{n,m}(K))$ and $X_n(Q_{n,m}(K))$ form an exotic pair for any integer $n<0$. 
    \item If $\varepsilon(K)=0$, $X_n(P_{n,0}(K))$ and $X_n(Q_{n,0}(K))$ form an exotic pair for any integer $n\neq 0$. 
    \item If $\varepsilon(K)=1$, then $X_n(P_{n,m}(K))$ and $X_n(Q_{n,m}(K))$ form an exotic pair for any integer $n<2\tau(K)$. 
\end{enumerate}
\end{theorem}

In particular, Theorem~\ref{hom-wan-cases} covered cases that were not known to be exotic by the previous approaches, see \cite[Remark 1.4]{homwan2026}. Hom and Wan additionally posit the following conjecture about exactly which knot traces form exotic pairs; they expect that for every non-trivial case, the resulting knot traces are always exotic.
\begin{conjecture}\label{hom conjecture}(\cite[Conjecture 1.5]{homwan2026}) Let $K$ be any knot in $S^3$, and $m, n$ integers. Then $X_n(P_{n,m}(K)$) and $X_n(Q_{n,m}(K))$
form an exotic pair, except when $K$ is the unknot and $m = n = 0$.
\end{conjecture}
 \begin{figure}
        \centering
        \includegraphics[width=0.75\linewidth]{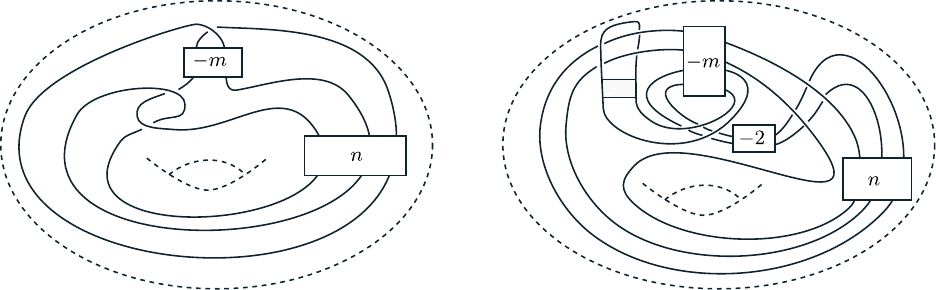}
        \caption{The satellite patterns $P_{n,m}$ (left) and $Q_{n,m}$ (right). The shaded band shows $Q_{n,m}$ is concordant to a longitude of $S^1\times D^2$. Here the box represents the number of full-twists, with sign.}
        \label{fig: pnm qnm}
    \end{figure}
    
 In this paper, we apply Chen-Hanselman \cite{chen-hanselman-2026-immersed} immersed curves methods, building on those developed by Hanselman, Rasmussen, and Watson \cite{hrw2024bordered}, to identify new cases where $X_n(P_{n,m}(K))$ and $X_n(Q_{n,m}(K))$ form an exotic pair. Specifically, we compute the following $\varepsilon$ invariants:

\begin{theorem}
\label{main theorem} Let $K$ be a knot in $S^3$. \begin{enumerate}
    \item If $\varepsilon(K)=0, n>0$, and $m\geq 1$, $\varepsilon(P_{n,m}(K))=1$.
    \item If $\varepsilon(K)=0, n > 0$, and $m\leq-2$, $\varepsilon(P_{n,m}(K))=1$.
    \item If $\varepsilon(K)=0, n < 0$, and $m\leq-2$, $\varepsilon(P_{n,m}(K))=-1$.
\end{enumerate}
\end{theorem}

This allows us to make partial progress towards Conjecture~\ref{hom conjecture} by showing the following:
\begin{corollary}
\label{main corollary}Let $K$ be a knot in $S^3$. \begin{enumerate}
    \item If $\varepsilon(K)=0, n>0$, and $m\geq 1$, $X_n(P_{n,m}(K))$ and $X_n(Q_{n,m}(K))$ form an exotic pair.
    \item If $\varepsilon(K)=0, n \neq 0$, and $m\leq-2$, $X_n(P_{n,m}(K))$ and $X_n(Q_{n,m}(K))$ form an exotic pair.
\end{enumerate}
\end{corollary}

The cases in Corollary~\ref{main corollary} are complementary to those in Theorem~\ref{hom-wan-cases}. Moreover, as far as the authors are aware, no previous examples of $X_n(P_{n,m}(K))$ and $X_n(Q_{n,m}(K))$ with $m < 0$ (for any knot $K$ and integer $n$) were known to form an exotic pair (see \cite[Remark 1.4]{homwan2026}). 

This result, combined with cases (1) and (2) in Theorem~\ref{hom-wan-cases} finishes all $(n, m)$ cases for $\varepsilon(K)=0, m\neq -1$, and $n\neq 0$ -- indeed, all such cases give us an exotic pair of trace manifolds. 

\begin{corollary}
    Let $K$ be a knot in $S^3$. If $\varepsilon(K)=0$, $n\neq 0$, and $m\neq -1$, then $X_n(P_{n,m}(K)$) and $X_n(Q_{n,m}(K))$
form an exotic pair.
\end{corollary}

\subsection{Outline of the proof}
$Q_{n,m}$ is concordant to a longitude of the solid torus (see the shaded band in Figure~\ref{fig: pnm qnm}), so $Q_{n,m}(K)$ is concordant to $K$. In particular, $ \varepsilon (Q_{n,m}(K))=\varepsilon(K)$. It thus suffices to show that for the cases detailed in Theorem~\ref{main theorem}, $\abs{\varepsilon(P_{n,m}(K))}$ differs from $\abs{\varepsilon(K)}$. In \cite{homwan2026}, Hom and Wan used the crossing change inequality of the $\tau$-invariant and its relation to the $\varepsilon$- and $\nu$-invariants to prove Theorem~\ref{hom-wan-cases}. The crossing change inequality states that $\tau(P_{n,m+1}(K))\geq \tau(P_{n,m}(K))$, which can then be used to obtain bounds on $\tau(P_{n,m}(K))$. In case (1) of Theorem~\ref{hom-wan-cases}, Hom and Wan use this to obtain $\tau(P_{n,m}(K))\geq 1$, which implies that $\abs{\varepsilon(P_{n,m}(K))}=1\neq\abs{\varepsilon(Q_{n,m}(K))}$. In (2), $\varepsilon(P_{n,0}(K))$ is shown to be 1 directly using a result of Bodish \cite{Bodish-twisted-Mazur}. Finally, in (3), the crossing change inequality is used to bound $\nu(P_{n,m}(K))$ and show it differs from $\nu(Q_{n,m}(K))$. 

In the $n>0$ cases of Corollary~\ref{main corollary}, we obtain that $\nu(P_{n,m}(K))=\nu(Q_{n,m}(K))$; see Table~\ref{tab:epsilon 0} in Section~\ref{section: conclusion}. In the $n<0, m\leq -2$ case, we are decreasing $m$ compared to the base case of $P_{n,0}$, so the crossing change inequality does not give a useful bound. In all our cases, we  obtain $\tau(P_{n,m}(K))=0$, so cannot conclude anything about the $\abs{\varepsilon}$ invariant from $\tau$. 

We thus rely on finding the $\abs{\varepsilon}$ invariant directly to show $X_n(P_{n,m}(K))$ and $X_n(Q_{n,m}(K))$ are not diffeomorphic. We do so by using the immersed curve techniques from bordered Heegaard Floer homology. See \cite{hanselman2023cabling, Chen-2023, Patwardhan-Xiao, Bodish-twisted-Mazur} for related prior work on this. 

\begin{remark}
    Note that $m=-1$ is not included in our cases; for $m=-1$, the pattern $P_{n,m}$ is always concordant to the standard longitude, as shown in Figure~\ref{fig:m=-1 concordance}. Thus for $m=-1$, any invariant of the $n$-trace that is also a knot concordance invariant of the underlying knot will not be sufficient to distinguish $X_n(P_{n,m}(K)$) from $X_n(Q_{n,m}(K))$. We also do not consider the case of $\varepsilon(K)=0, n=0$. In this case, on one hand, since $\varepsilon(K)=0$, our concordance invariants do not differentiate between $K$ and the standard unknot $U$. On the other hand, $P_{0,m}$ is an unknotted pattern in $S^3$ and $Q_{0,m}$ is concordant to an unknotted pattern in $S^3$ (via the shaded band in Figure~\ref{fig: pnm qnm}), so we obtain $\varepsilon(P_{0,m}(K))= \varepsilon(Q_{0,m}(K))= 0$. 
\end{remark}
\subsection*{Organization}
In Section~\ref{section: alpha}, we build the first immersed curve of the construction, $\alpha_K$, associated to the companion knot $K$ with framing $n$. In Section~\ref{section: heegaard}, we construct the Heegaard diagram for $P_{0,m}$, and in Section~\ref{section: beta}, we build the second immersed curve, $\beta$, associated to $P_{0,m}$. In Section~\ref{section: pairing}, we combine the immersed curves and compute $\abs{\varepsilon(P_{n,m}(K))}$, proving Theorem~\ref{main theorem}. In Section~\ref{section: conclusion}, we conclude the proof of Corollary~\ref{main corollary} and discuss the scope of our results. 
\subsection*{Acknowledgments} The authors thank Abhishek Mallick for pointing them towards this problem and for helpful discussions. We also thank Ina Petkova for helpful conversations, and Jennifer Hom and Wenzhao Chen for useful comments on an earlier draft.
\begin{figure}
    \centering
    \includegraphics[width=0.75\linewidth]{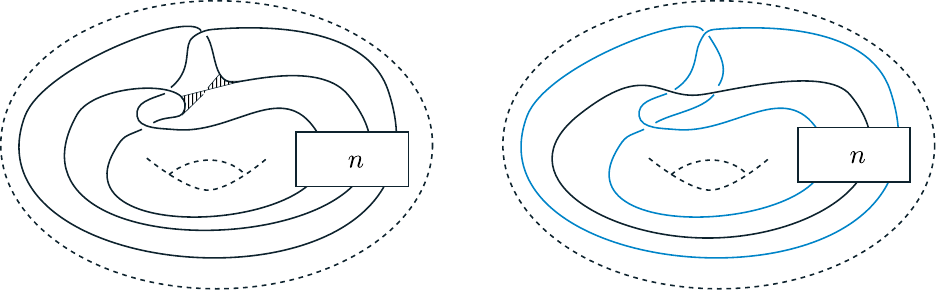}
    \caption{The shaded band shows a concordance between $P_{n,-1}$ and the longitude of the torus: the blue component can be isotoped to a homologically trivial unknot inside $S^1\times D^2$.}
    \label{fig:m=-1 concordance}
\end{figure}

\section{The $\alpha_K$ curve for the $n$-framed companion knot}\label{section: alpha} 

We assume that the reader is familiar with the concepts of bordered Heegaard Floer homology and its immersed curve interpretation. For the sake of establishing notation, we include a brief background below. Lipshitz, Ozsváth, and Thurston in \cite{LOT-bordered-floer} developed bordered Heegaard Floer homology, which associates differential modules (type $D$-structures) and $\mathcal{A}_{\infty}$ modules (type $A$-structures) to $3$-manifolds with boundary. Hanselman, Rasmussen, and Watson \cite{hrw2024bordered} introduced a geometric interpretation of bordered Heegaard Floer homology, translating type $D$ and $A$ structures into immersed curves, where the pairing between these structures translates to counting the intersection points between the immersed curves. Chen and Hanselman \cite{chen-hanselman-2026-immersed} describe a method for computing the knot Floer complex $\mathit{CFK}^-_\mathcal{R}(P(K))$ of a satellite knot $P(K)$ over the ring $\mathcal{R}=\mathbb{F}_2[U,V]/UV$ using immersed curves, in which the pattern knot and the companion knot are translated to immersed curves that are then paired. In this section, we describe the immersed curve associated to the companion knot, denoted $\alpha_K$ (typically depicted in red). The curve associated to the pattern knot is denoted $\beta$ (typically depicted in blue). 

We consider the case of companion knots $K$ with framing $n$ (which we denote by $K_{[n]}$) with $\varepsilon(K)=0$. In this case, $\tau(K)=0$ so the sign of $2\tau(K)-n$ is determined by the sign of $n$. It follows from \cite[Theorem 2.6]{Levine-nonsurjective} and \cite[Theorem 11.26]{LOT-bordered-floer} that the unstable chain in the type $D$-structure becomes 
\[\begin{tikzcd}[cramped,column sep=small]
	{\xi_0} && {\mu_1} && {...} && {\mu_n} && {\eta_0}
	\arrow["{\rho_{123}}", from=1-1, to=1-3]
	\arrow["{\rho_{23}}", from=1-3, to=1-5]
	\arrow["{\rho_{23}}", from=1-5, to=1-7]
	\arrow["{\rho_2}", from=1-7, to=1-9]
\end{tikzcd}\]
when $n>0$ and 
\[\begin{tikzcd}[cramped,column sep=small]
	{\eta_0} && {\mu_1} && \dots && {\mu_{\vert{n}\vert}} && {\xi_0}
	\arrow["{\rho_3}", from=1-1, to=1-3]
	\arrow["{\rho_{23}}", from=1-3, to=1-5]
	\arrow["{\rho_{23}}", from=1-5, to=1-7]
	\arrow["{\rho_1}"', from=1-9, to=1-7]
\end{tikzcd}\]
when $n<0$. Translating this to immersed train tracks as in \cite{hrw2024bordered} (see also \cite[Figure 13]{chen-hanselman-2026-immersed}), we get Figure~\ref{fig:alpha train tracks}.

\begin{figure}
    \centering
    \includegraphics[width=0.75\linewidth]{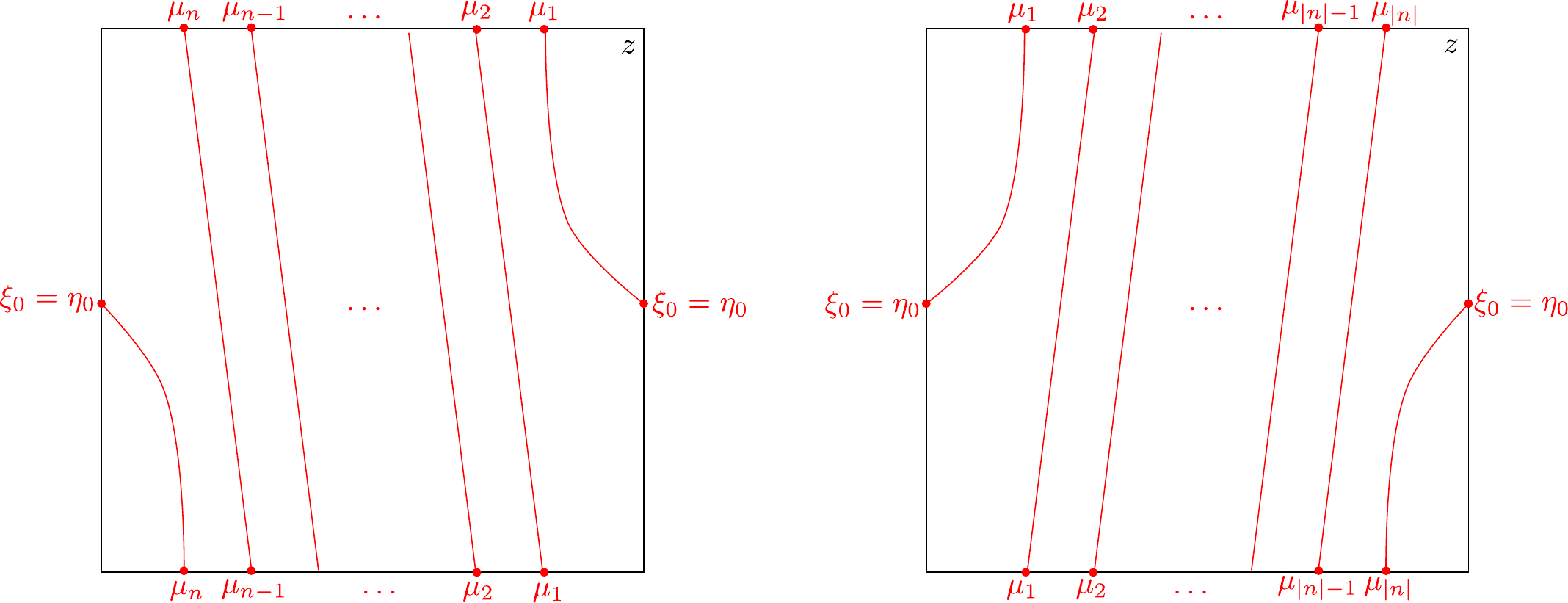}
    \caption{The immersed train tracks obtained from the unstable chain when $\varepsilon(K)=0$. Left: $n>0$, right: $n<0$.}
    \label{fig:alpha train tracks}
\end{figure}

We note there may be other elements of the type $D$-structure, but in this case the unstable chain gives us the principal $\alpha_K$ curve. After lifting to the universal cover, this corresponds to an infinite curve, with a vertical segment of height $n$ in each column. See Figure~\ref{fig:alpha curve} for the curve obtained for each sign of $n$.

\begin{figure}
    \centering
    \includegraphics[width=0.75\linewidth]{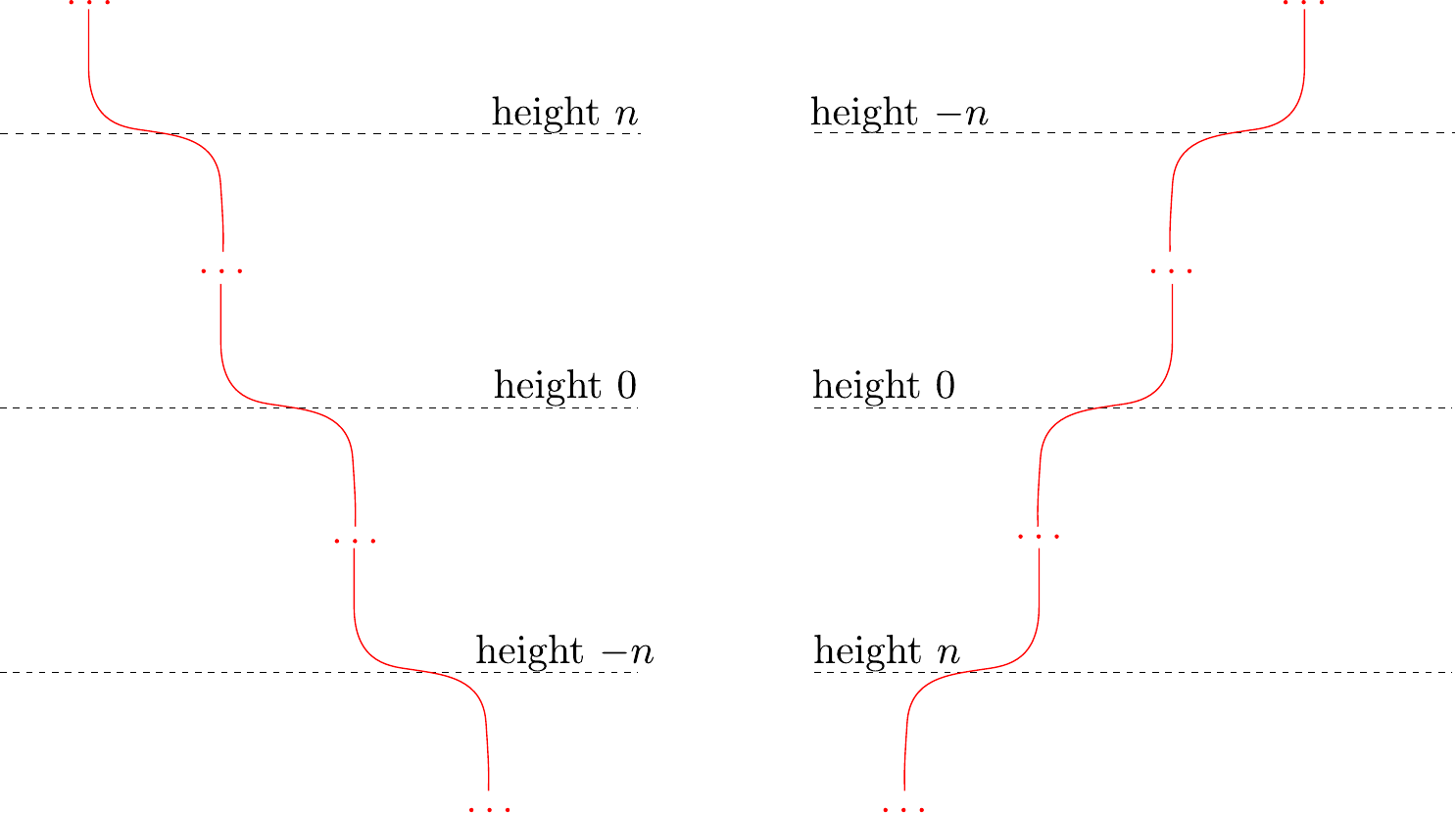}
    \caption{The shape of the $\alpha_K$ curve when $\varepsilon(K)=0$. Left: $n>0$, right: $n<0$.}
    \label{fig:alpha curve}
\end{figure}

\section{The Heegaard diagram for $P_{0, m}$}\label{section: heegaard}

We now want to construct the immersed curve $\beta$ of the pattern knot $P_{0,m}$. Instead of looking directly at $P_{n, m}(K)$, we transfer the $n$ full meridional twists to the framing of the companion knot and work with $P_{0,m}(K_{[n]})$, where we again use $K_{[n]}$ to mean the knot $K$ with framing $n$, instead of $P_{n,m}(K)$. The resulting knots are isotopic so all invariants remain the same. It thus suffices for us to consider the pattern $P_{0, m}$ and its bordered Heegaard diagram, which is a doubly-pointed Heegaard diagram with basepoints labeled $z$ and $w$. 

We first observe that the pattern knot $P_{0, m}$ together with the meridian of the solid torus forms a $2$-bridge link. We start by computing its Conway form. The top isotopy in Figure~\ref{fig:m-both-isotopy} shows that for $m \geq 0$, \[P_{0,m}\cup \text{meridian of } S^1\times D^2=C(2, 2m, 2, 1, 2)\] and the bottom isotopy in Figure~\ref{fig:m-both-isotopy} shows that for $m<0$,  \[P_{0,m}\cup \text{meridian of } S^1\times D^2 = C(-2, 2m+1, 2, 1, 2)\]

\begin{figure}
    \centering
    \includegraphics[width=0.95\linewidth]{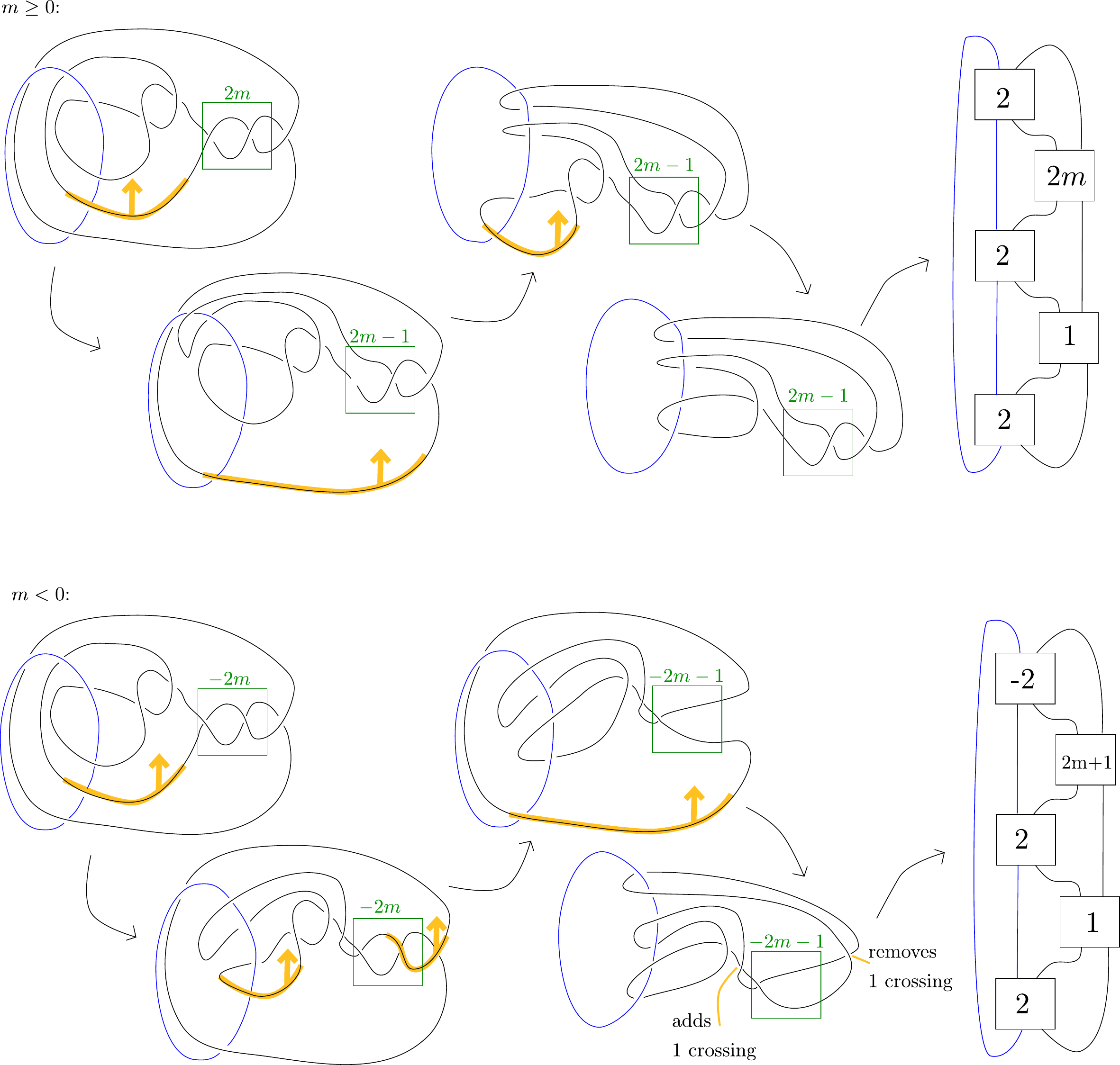}
    \caption{Isotopy of $P_{0,m}$ into Conway form, for $m\geq 0$ (above) and $m< 0$ (below). The strand that moves at each step is highlighted, with the direction of movement indicated. Green numbers indicate the number of crossings inside the green box.}
    \label{fig:m-both-isotopy}
\end{figure}

Using the Conway form, we translate our knots into Schubert normal form (see \cite[Section 12]{burde-zieschang-2003-knots}), then use this to build Heegaard diagrams for our pattern knots $P_{0,m}$. In particular, the family $P_{0,m}$ are examples of what are called \textit{$(1, 1)$-patterns}, see \cite[Definition 1.1]{Chen-2023}. By \cite[Section 5.1]{Chen-2023}, these patterns admit Heegaard diagrams $\mathcal{H}(r, s)$ with $r,s$ integers such that $r \neq 0$, $s \geq 0$, and $\gcd(2r-1, s+1) = 1$. These $\mathcal{H}(r,s)$ diagrams consist of the following elements (see Figure~\ref{fig:heegaard-diagram} for the diagrams of $P_{0,m}$).

\begin{enumerate}
    \item A $z$-basepoint is placed near the left edge of the diagram and a $w$-basepoint is placed near the right edge. Note our labeling convention for the basepoints is the opposite of the one in \cite{Chen-2023}.
    \item There are $\abs{r}$ arcs around each basepoint starting and ending at the same vertical edge of the Heegaard diagram (referred to as \textit{rainbows}, part of the \textit{main region} of the diagram).
    \item $s$ segments slant across the diagram. If $r>0$, they slant downwards, starting from the left edge of the diagram above the $z$-basepoint and ending below the $w$-basepoint on the right edge. If $r<0$, they slant upwards, starting from the left edge below the $z$-basepoint and ending above the $w$-basepoint. These segments are referred to as \textit{strands}, also elements of the main region of the diagram.
    \item Additional segments and arcs are placed above the elements in the main region, added to ensure the intersection numbers with the edges of the diagram satisfy certain properties -- see below. This part of the diagram is referred to as the \textit{dependent region}).
\end{enumerate}

The strands and rainbows join to become one arc, called the \textit{main arc} (denoted $l_m$), and the dependent region strands join to create one arc, called the \textit{dependent arc} (denoted $l_d$). The $\beta$ curve of the Heegaard diagram is the union of arcs $l_m \cup l_d$. Finally, the horizontal edge of the Heegaard diagram is denoted $\lambda$ and and the vertical is $\mu$. It is required that the algebraic intersection numbers satisfy $(l_m \cup l_d) \cdot \lambda = 1$ and $(l_m \cup l_d) \cdot \mu = 0$ to produce an admissible diagram. We denote such a Heegaard diagram as $\mathcal{H}(r, s)$. 

\begin{figure}[h!]
        \centering
        \includegraphics[width=\linewidth]{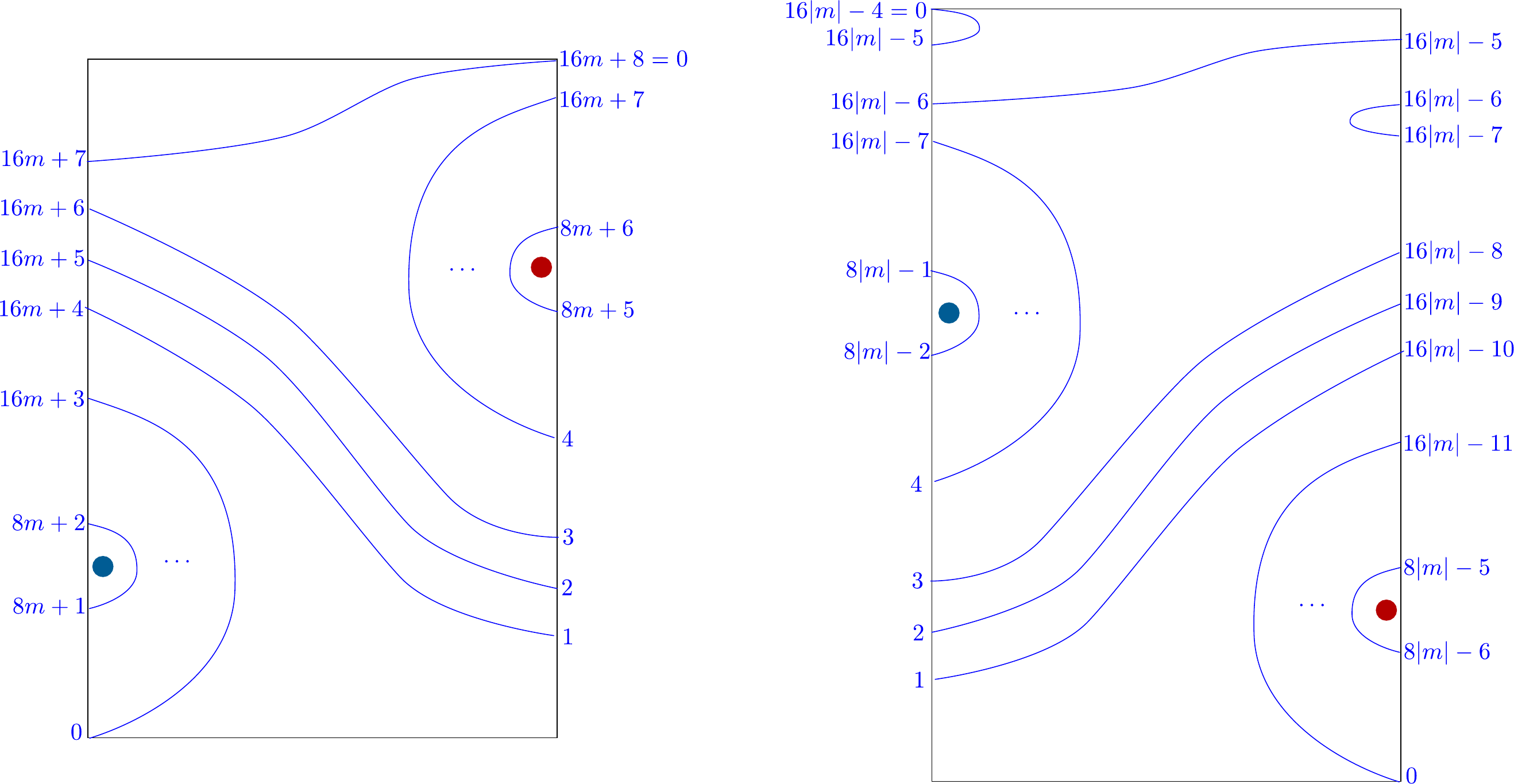}
        \caption{The Heegaard diagrams for $P_{0,m}$. Left: $m\geq 0$, right $m<0$. The blue dots represent $z$-basepoints, while the red dots represent $w$-basepoints.}
        \label{fig:heegaard-diagram}
    \end{figure}

Using the correspondence between Conway form and the Schubert normal form of two-bridge links (see \cite[Prop. 12.13]{burde-zieschang-2003-knots}), we compute the Schubert normal form $b(p,q)$ of $P_{0,m}$ for $m\geq0$:
\[\frac{q}{p}=\frac{1}{2+\frac{1}{2m+\frac{1}{2+\frac{1}{1+\frac{1}{2}}}}}=\frac{16m+3}{32m+14}\]
For $m<0$, this becomes 
\[\frac{q}{p}=\frac{1}{-2+\frac{1}{2m+1+\frac{1}{2+\frac{1}{1+\frac{1}{2}}}}}=\frac{16m+11}{-32m-14}\]
We now use \cite[Theorem 5.4]{Chen-2023} to determine $r$ and $s$ in the associated Heegaard diagram $\mathcal{H}(r, s)$:
\[\begin{cases}
    \operatorname{sgn}(r)(2\abs{r}-1)=q\\
    2s+4\abs{r} = p 
\end{cases}\implies \begin{cases}
    m\geq 0: r = 8m+2, s = 3\\
    m<0: r = -(8\abs{m}-5), s=3
\end{cases}\]
As a check, note that that the Heegaard diagram for $m=0$ is $\mathcal{H}(2, 3)$, which recovers the Mazur pattern, and the diagram for $m=-1$ is $\mathcal{H}(-3, 3)$, which recovers the Mazur pattern with one flipped crossing. 

To finish building $\mathcal{H}(r, s)$, we calculate how many strands are in the dependent region. By \cite[Section 5]{Chen-2023}, for $\mathcal{H}(r, s)$ to be an admissible diagram, we require $l_d$ to satisfy $(l_m \cup {l_d}) \cdot \mu = 0$. We use \cite[Proposition 6.2]{Chen-2023}, which states that the algebraic intersection number of $l_m$ and a small push-off of $\mu$ is given by the following formula:

\begin{equation*}
    l_m \cdot \mu = 1 - \underset{i=1}{\overset{2r+s-1}{\sum}} (-1)^{\lfloor \frac{i(2r-1)}{2r+s} \rfloor}
\end{equation*}

We can calculate this for either sign of $m$. For $m > 0 $ we evaluate 
\begin{align*}
    \underset{i=1}{\overset{2r+s-1}{\sum}} (-1)^{\lfloor \frac{i(2r-1)}{2r+s} \rfloor} &
    = \underset{i=1}{\overset{16m+6}{\sum}} (-1)^{\lfloor \frac{i(16m+3)}{16m+7} \rfloor}\\
    &= \underset{i=1}{\overset{4m+1}{\sum}} (-1)^{i-1} 
    +  \underset{i=4m+2}{\overset{8m+3}{\sum}} (-1)^{i-2} 
    + \underset{i=8m+4}{\overset{12m+5}{\sum}} (-1)^{i-3} 
    + \underset{i=12m+6}{\overset{16m+6}{\sum}} (-1)^{i-4} \\
    &= 1+ 0+0+1 = 2.
\end{align*} 

Thus, $l_m \cdot \mu = -1$ in the $m\geq 0$ case. For $m < 0$, we have 
\begin{align*}
    \underset{i=1}{\overset{2r+s-1}{\sum}} (-1)^{\lfloor \frac{i(2r-1)}{2r+s} \rfloor} 
    &= \underset{i=1}{\overset{16|m|-8}{\sum}} (-1)^{\lfloor \frac{i(16|m|-11)}{16|m|-7} \rfloor} \\
    &= \underset{i=1}{\overset{4|m|-2}{\sum}} (-1)^{i-1}
    + \underset{i= 4|m|-1}{\overset{8|m|-4}{\sum}} (-1)^{i-2} + \underset{i= 8|m|-3}{\overset{12|m|-6}{\sum}} (-1)^{i-3} 
    + \underset{i= 12|m|-5}{\overset{16|m|-8}{\sum}} (-1)^{i-4}\\ 
    &= 0+0+0+0=0
\end{align*}
so $l_m \cdot \mu = 1$ for the negative case. Thus, to satisfy $\beta \cdot \mu = 0$, we need one dependent region strand in each case. To make its orientation compatible with the orientation on $l_m$, in the $m<0$ case we add two small rainbows in the dependent region which do not encircle any basepoints. Our Heegaard diagram for $P_{0,m}$ therefore has the form shown in Figure~\ref{fig:heegaard-diagram} for either case. We will occasionally denote the Heegaard diagram of ${P}_{0,m}$ as $\mathcal{H}(P_{0, m})$ when the sign of $m$ is clear. 
\begin{figure}[h!]
    \centering
    \includegraphics[width=0.95\linewidth]{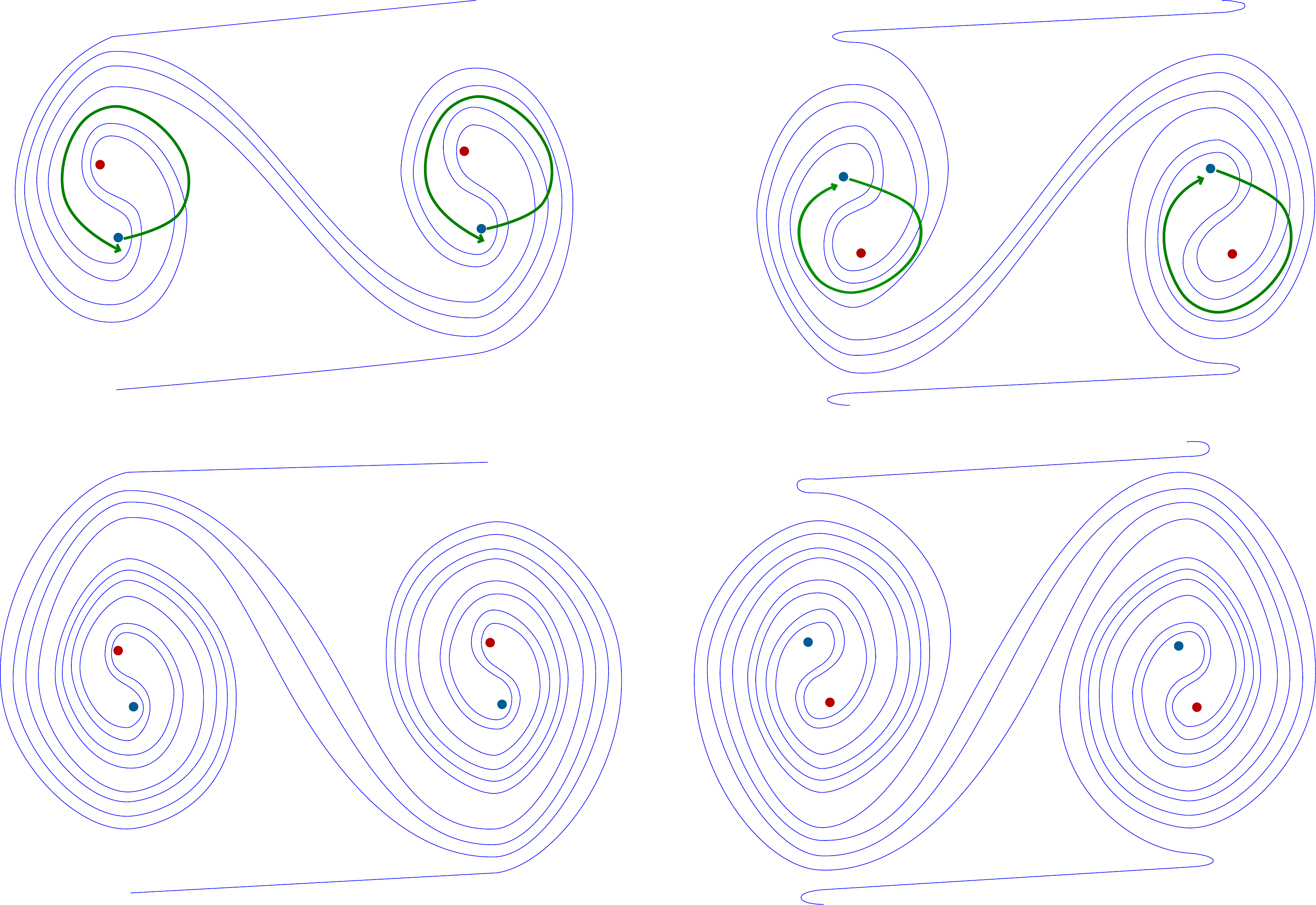}
    \caption{The $\beta$ immersed curve for $m=1$ (top left), $m=2$ (bottom left), $m=-2$ (top right), and $m=-3$ (bottom right). These are horizontal slices; the full curves extend infinitely up and down. We see the bottom curves can be obtained from the top curve by the winding move moving $z$ around $w$ along the green arrows (counterclockwise  if $m\geq 1$, clockwise if $m\leq -2$) and dragging the blue curves along, as described in Section~\ref{section: beta}. Note the dependent region strands for $m\leq -2$ are drawn as they appear in the Heegaard diagram, but they can be isotoped to be ``pulled tight".  }
    \label{fig:beta m 1 2}
\end{figure}

\section{The $\beta$ curve for $P_{0,m}$}\label{section: beta}

To obtain the $\beta$ curve for $P_{0,m}$, we take the universal cover of the Heegaard diagram, and look at a single lift of $l_m \cup l_d$. Doing this for small values of $\abs{m}$, we obtain the immersed curves drawn in Figure~\ref{fig:beta m 1 2}. Notice that visually, increasing $\abs{m}$ looks like the result of ``moving the $z$-basepoint around the $w$-basepoint", either counterclockwise if $m\geq 1$ or clockwise if $m\leq -2$. We will use \emph{winding move} to mean taking the blue {$z$}-basepoint and moving it along the green arrows indicated in Figure~\ref{fig:beta m 1 2}. During this move, the relevant blue $\beta$-arcs move as in a finger-pushing isotopy. For a visual, compare the top images of Figure~\ref{fig:beta m 1 2} with the bottom images of Figure~\ref{fig:beta m 1 2}.

\begin{figure}[h!]
        \centering
        \includegraphics[width=\linewidth]{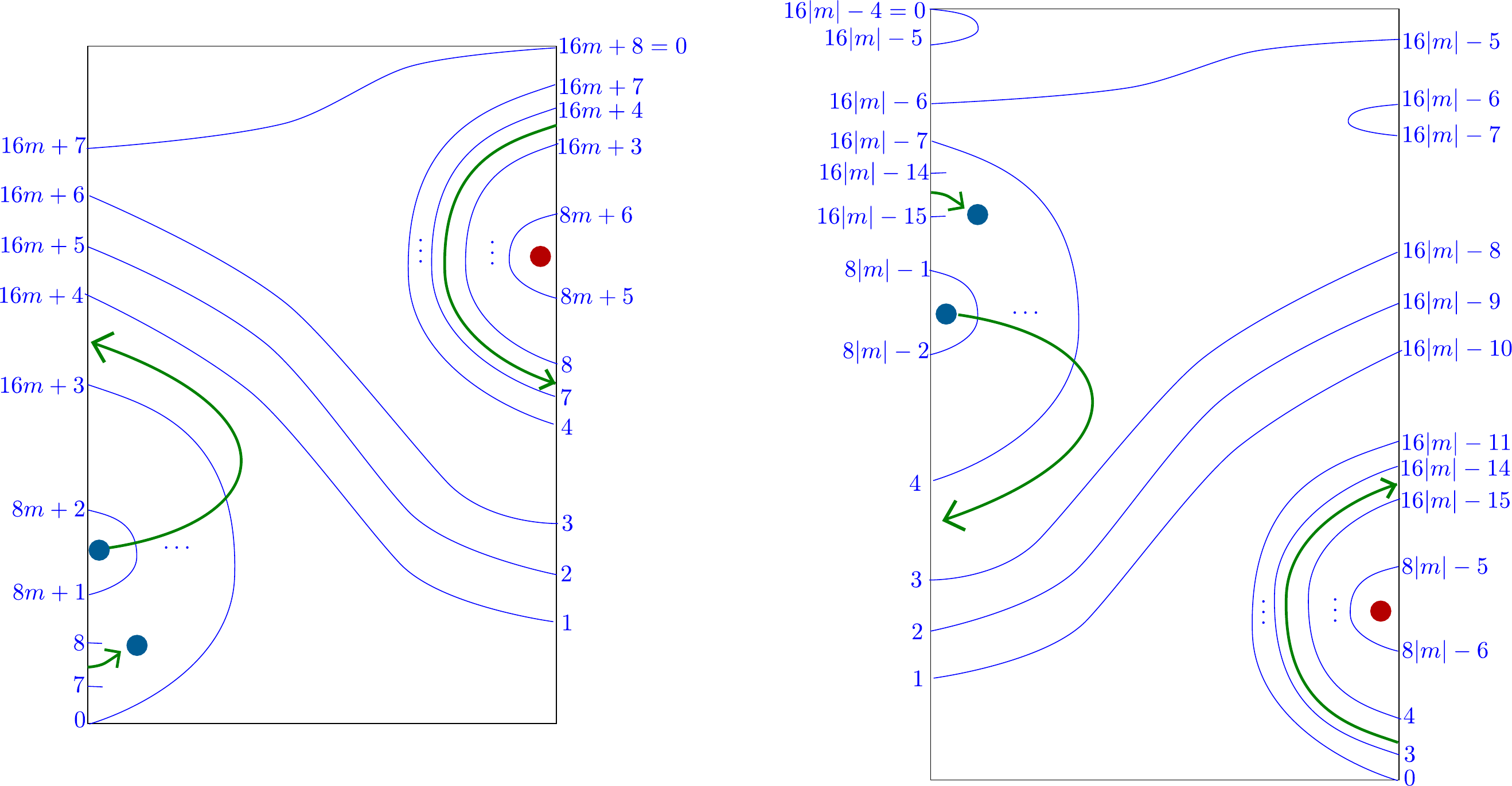}
        \caption{The green arrows represent the path that $z$ traces out as it winds around $w$. Left: $m\geq 1$, right: $m\leq-2$.}
        \label{fig:z-path}
    \end{figure}

To make this rigorous, we verify that the counterclockwise winding move indeed takes the Heegaard diagram for $P_{0,m}$ to the Heegaard diagram for $P_{0, m+1}$ when $m\geq 1$, while the clockwise winding move takes the Heegaard diagram for $P_{0,m}$ to the Heegaard diagram for $P_{0,m-1}$ when $m\leq -2$, and so the shape of the immersed curve moves accordingly.

\begin{theorem}
    Let $\mathcal{H}(P_{0,m})$ be the Heegaard diagram associated to the knot $P_{0, m}$ described in Section~\ref{section: heegaard}. Winding the $z$-basepoint once around the nearest $w$-basepoint (counterclockwise when $m \geq 1$ and clockwise when $m \leq -2$) and pulling along the strands of the $\beta$ curve, we obtain the Heegaard diagram $\mathcal{H'}$ as follows:
    \begin{enumerate}
        \item When $m \geq 1$, $\mathcal{H}(P_{0,m}) = \mathcal{H}(8m+2, 3)$ is sent to $\mathcal{H}'=\mathcal{H}(P_{0,m+1}) = \mathcal{H}(8(m+1)+2, 3)$. 
        \item When $m \leq -2$, $\mathcal{H}(P_{0,m}) = \mathcal{H}(-(8|m|-5), 3)$ is sent to $\mathcal{H}'=\mathcal{H}(P_{0,m-1}) = \mathcal{H}(-(8|m-1|-5), 3)$. 
    \end{enumerate} 
\end{theorem}
    
\begin{proof}
    Consider the case of $m\geq 1$. The corresponding Heegaard diagram is  $\mathcal{H}(P_{0,m})$ with $r = 8 m + 2$ rainbows and $s=3$ strands, and one strand in the dependent region. We lay the Heegaard diagram over a $1\times (16 m + 8)$ grid, where $\beta$ intersects the vertical edges of $\mathcal{H}$ at integer coordinates. This allows us to specify the rainbows and strands by their endpoint coordinates. For the case of $m \leq -2$, the corresponding Heegaard diagram $\mathcal{H}(P_{0,m})$ has $|r| = 8 |m| - 5$ rainbows and $s=3$ strands, also with one strand and two arcs in the dependent region. We lay this over a $1\times (16 |m| - 4)$ grid. See Figure~\ref{fig:heegaard-diagram} for a visual with the indexing.

   On the $m\geq 1$ diagram, the three main region strands have endpoints $(0,16m+3+i)$ and $(1, i)$ for $i=1,2,3$. $8m+2$ rainbows encircle $z$ on the left side of the diagram ($z$ is placed in between heights $8m+1$ and $8m+2$) and $8m+2$ rainbows encircle $w$ on the right side of the diagram ($w$ is placed between heights $8m+5$ and $8m+6$). Finally, the dependent region strand connects $(0, 16m+7)$ and $(1, 16m+8)$.

    On the $m \leq -2$ diagram, the three main region strands have endpoints $(0, i)$ and $(1,16|m|-11+i)$ for $i=1,2,3$. $8|m|-5$  rainbows encircle $z$ on the left side of the diagram ($z$ is placed in between heights $8|m|-2$ and $8|m|-1$ in this case) and $8|m|-5$ rainbows encircle $w$ on the right side of the diagram ($w$ is placed between heights $8|m|-6$ and $8|m|-5$). Finally, the dependent region strand connects $(0, 16|m|-6)$ and $(1, 16|m|-5)$. In order to connect the dependent region strand to the main region arc, include two extra arcs in the dependent region, one with endpoints $(0, 16|m|-5)$ and $(0, 16|m|-4)$ and the other with endpoints $(1, 16|m|-7)$ and $(1, 16|m|-6)$. For a visual of both cases, see Figure~\ref{fig:heegaard-diagram}. 

    We track the winding behavior of $z$ around $w$ in steps (specifically Steps 1, 2, and 3 below), noting newly created rainbows and strands. We apply an isotopy (Step 4) to the resulting curve to simplify, and show it produces the Heegaard diagram for $\mathcal{H}(P_{0, m+1})$ for  $m \geq 1$ and $\mathcal{H}(P_{0, m-1})$ for $m\leq -2$. 

    \textbf{Step 1: The path of $z$ first intersects the left edge of $\mathcal{H}$. } $z$ moves first between the outermost rainbow which intersects the left edge of the Heegaard diagram and the nearest strand which intersects that same edge. In the $m \geq 1$ case, this is between $(0, 16m+3)$ and $(0, 16m+4)$. In the $m \leq -2$ case, this is between $(0, 3)$ and $(0, 4)$. 
    
    As $z$ intersects the left edge, $z$ carries the $\abs{r}$ rainbows surrounding it, creating $\abs{r}$ new rainbows; in the $m \geq 1$ case, this is $8m+2$ new rainbows created via the top-half  of the original $8m+2$ rainbows being ``bent" to intersect the left edge of the Heegaard diagram. In the $m \leq -2$ case, this is $8|m|-5$ rainbows created via the bottom-half of the original $8|m|-5$ rainbows. This is illustrated in Figure~\ref{fig:z-path}. 
    
    \textbf{Step 2: $z$ winds around $w$ on the right side of $\mathcal{H}$.} $z$ re-emerges on the right-edge of the diagram. In the $m\geq 1$ case, $z$ re-emerges between $(1, 16m+3)$ and $(1, 16m+4)$ above $w$, between the top-endpoints of the inner and outermost rainbows. In the $m \leq -2$ case, $z$ emerges between $(1, 3)$ and $(1, 4)$, below $w$, between the bottom-endpoints of the inner and outermost rainbows.
    
    In both cases, $z$ then travels between two rainbows around $w$, intersecting the right-side edge of the graph once again. For $m\geq 1$, this is between $(1, 7)$ and $(1, 8)$. For $m \leq -2$, this is between $(1, 16|m|-15)$ and $(1, 16|m|-14)$. As $z$ moves, the $|r|$ rainbows continue to be dragged by it, resulting in $2\abs{r}$ rainbows around $w$. See Figure~\ref{fig:z-moves-around-w} for a visual of this path and the resulting new rainbows.

    \begin{figure}
        \centering
        \includegraphics[width=0.88\linewidth]{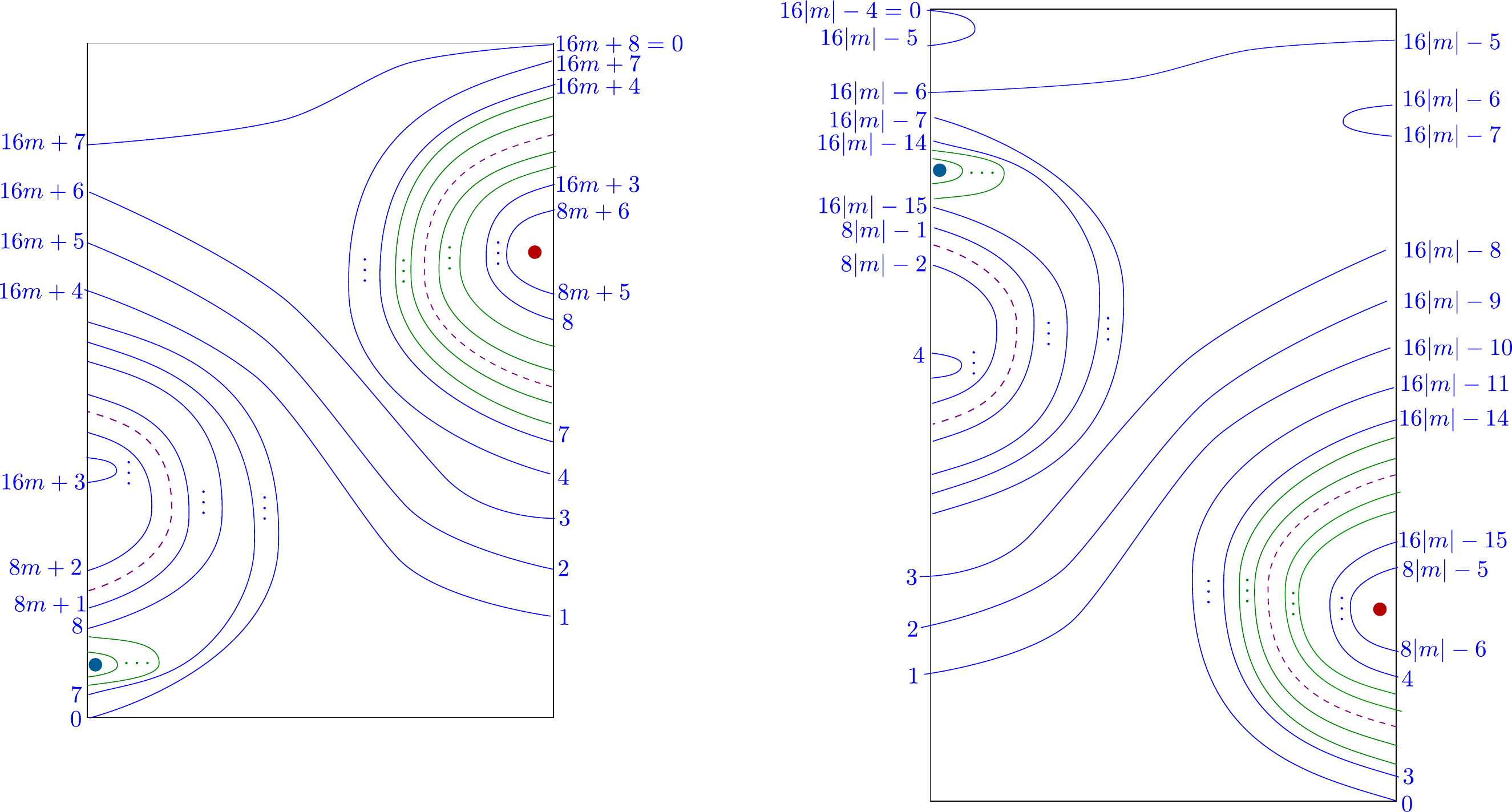}
        \caption{$z$ has finished winding around $w$. The path of $z$ is denoted by the dashed purple line. New rainbows are indicated in green. Left: $m\geq 1$, right: $m\leq -2$.}
        \label{fig:z-moves-around-w}
    \end{figure}

    \textbf{Step 3: $z$ re-emerges on the left side of $\mathcal{H}$.} $z$ re-emerges from the left edge, between $(0, 7)$ and $(0,8)$ for the $m \geq 1$ case and $(0, 16|m|-15)$ and $(0, 16|m|-14)$ in the $m \leq -2$ case. It caries $r$ rainbows around it, so in the $m \geq 1$ this results in an extra $8m+2$ rainbows between $(0, 7)$ and $(0,8)$ and in the $m \leq -2$ case, this is $8|m|-5$ rainbows between $(0, 16|m|-15)$ and $(0, 16|m|-14)$. These new rainbows are drawn in green on Figure~\ref{fig:z-moves-around-w}. This completes the movement of $z$.

    \begin{figure}
        \centering
        \includegraphics[width=0.88\linewidth]{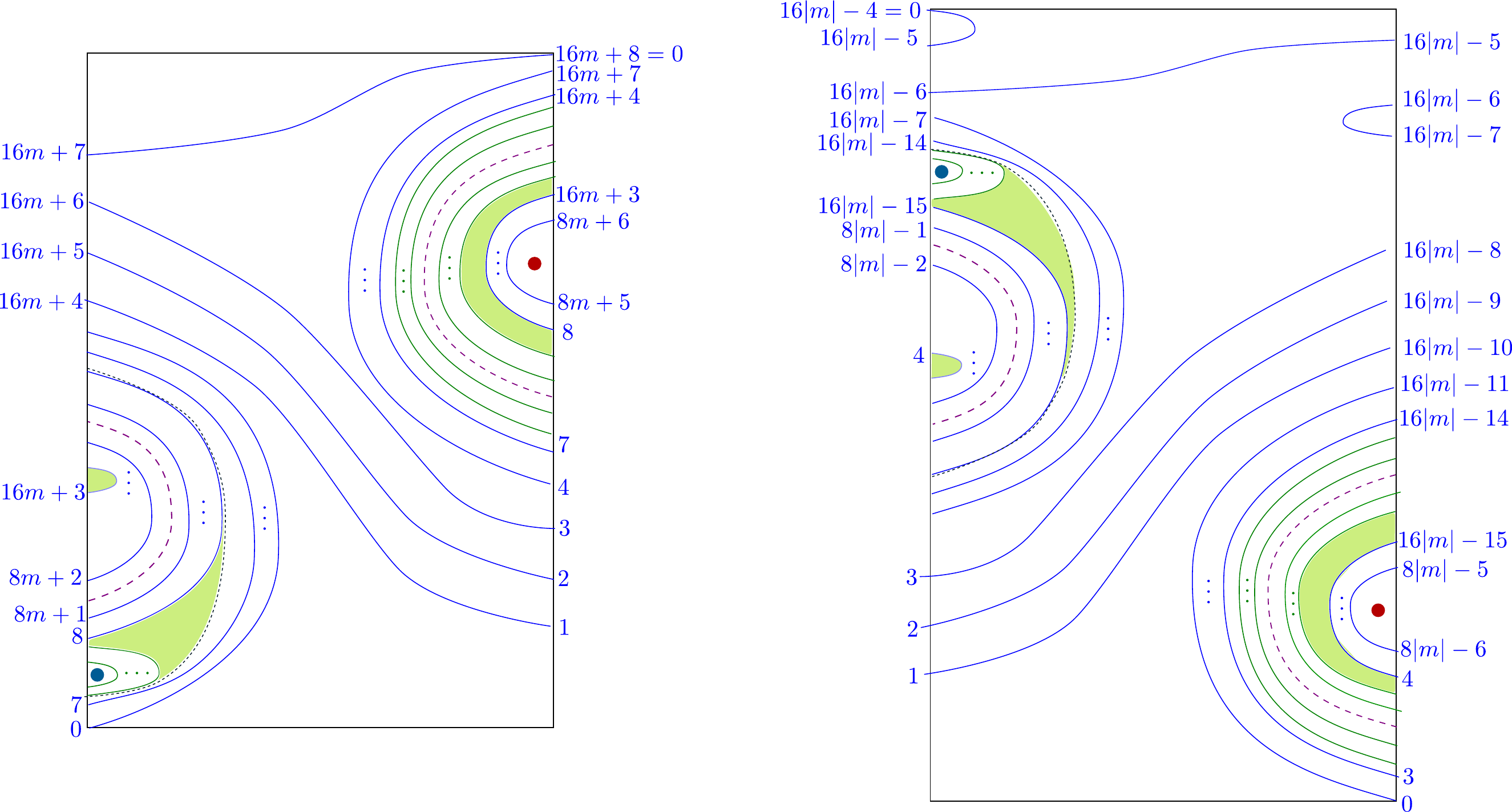}
        \caption{The process of isotoping away extra rainbows, highlighted in green. Left: $m\geq 1$, right: $m\leq -2$.}
        \label{fig:extra rainbows isotopy}
    \end{figure}

    \textbf{Step 4: An isotopy on the resulting $\beta$ curve.} The last step remaining is to isotope away redundant rainbows created by the movement of $z$. See Figure~\ref{fig:extra rainbows isotopy} for a visual of this process. Consider the innermost rainbow on the left edge that does not encircle a basepoint; label it $\mathscr{R}$. The green shaded region in Figure~\ref{fig:extra rainbows isotopy} is the area over which we isotope $\mathscr{R}$, and we describe that process as follows for each case separately. 
    
    In the $m \geq 1$ case the lower end of $\mathscr{R}$ sits at $(0, 16m+3)$. Since $\mathscr{R}$ does not encircle a basepoint, we can isotope it across the left edge where it will re-emerge on the right-hand side above $w$. $\mathscr{R}$ thus connects two rainbows around $w$ -- one with ends $(1, 16 m+3)$ and $(1, 8)$, and the innermost new rainbow created when $z$ wound around $w$ in Step 2. We isotope $\mathscr{R}$ around $w$, getting rid of these two rainbows. Finally, when $\mathscr{R}$ re-emerges on the left-hand side, we will be able to merge two rainbows into one: the blue rainbow with lower end at $(0, 8)$ and the outermost new green rainbow around $z$, added at Step 3. 

    In the $m \leq 2$ case the upper endpoint of $\mathscr{R}$ sits at $(0, 4)$. It re-emerges on the right-hand side below $w$, connecting the two rainbows around $w$ -- one with ends $(1, 4)$ and $(1, 16|m|-15)$, and the innermost new rainbow created when $z$ wound around $w$ in Step 2. The isotopy gets rid of  these two rainbows just as in the positive case. Finally, when $\mathscr{R}$ re-emerges on the left-hand side, we merge two rainbows into one: the blue rainbow with lower end at $(0, 16|m|-15)$ and the outermost new green rainbow around $z$, added at Step 3. 
    
    In both cases, the isotopy eliminates the two rainbows around $w$ as previously described, as well as two rainbows on the left edge (this includes the original rainbow we eliminated, and the blue rainbow that was merged with the outermost green rainbow). However, the number of rainbows encircling $z$ did not change -- only one was modified, but it remained encircling $z$ after merging. We can now repeat this isotopy move with the next-innermost rainbow on the left side that does not encircle a basepoint. 

    \textbf{Step 5: Counting the number of rainbows and strands remaining.} For $m\geq 1$, we started Step 4 with a total of $16m+3-7 = 16m-4$ rainbows that do not encircle a basepoint on the left edge, $8m+10$ rainbows that do encircle $z$ ($8m+2$ green ones added during Step 3, and $8$ blue rainbows with one end below $z$'s final location). On the right side, we ended with $24m+6$ rainbows around $w$ after Step 3: $8m+2$ original rainbows and $16m+4$ rainbows added during the winding in Step 2. Repeating the isotopy move $8m-2$ times, we eliminate all $16m-4$ non-basepoint rainbows on the left and eliminate $16m-4$ added rainbows around $w$. 
    
    This leaves us with $8m+10$ rainbows around $z$ and $w$ each. The isotopy moves did not touch the strands in either the main or dependent regions, so we have produced a Heegaard diagram containing $r = 8m+10=8(m+1)+2$ rainbows around each basepoint, $s = 3$ strands, and one dependent region strand -- exactly $\mathcal{H}(P_{0,m+1})$  for $m \geq 1$.

    For the $m \leq -2$ case, we started Step 4 with a total of $16|m|-18$ rainbows that do not encircle a basepoint on the left edge, $8|m|+3$ rainbows that do encircle $z$ ($8|m|-5$ rainbows dragged along with $z$ depicted in green, and $8$ rainbows around $z$ in blue that originally existed around it). On the right side, we had $8|m|-5$ original (blue) rainbows and $16|m|-10$ added (green) around $w$, for a total of $24|m|-15$ at the end of Step 3. Repeating the isotopy move $8|m|-9$ times, we eliminate all $16|m|-18$ non-basepoint rainbows on the left, retain the $8|m|+3$ rainbows around $z$, and eliminate $16|m|-18$ rainbows around $w$ added in Step 2. 
    
    This leaves us with $8|m|+3$ rainbows around each $z$ and $w$ respectively. The isotopy again did not touch the strand or arcs in the dependent region, so those are preserved. This algorithm thus produced a Heegaard diagram containing $|r| = 8(|m|+1)-5$ rainbows and $s=3$ strands, and a single strand in the dependent region with the corresponding rainbows -- exactly $\mathcal{H}(P_{0, m-1})$ for $m \leq -2$, as desired.
\end{proof}

Visualizing these diagrams now on a universal cover, we see that if we restrict ourselves to one copy of the immersed curve $\beta$, the winding move acts on $\beta$ by making the curve wind further around the basepoints, one additional time for each winding move. See Figure~\ref{fig:beta m 1 2} for the immersed $\beta$ curves corresponding to $m=1,2$ and $m= -2, -3$. 

\section{The pairing diagram and computing $\abs{\varepsilon}$} \label{section: pairing}

We have created the $\beta$ and $\alpha_K$ curves for the pattern and companion knot respectively; we now create the pairing diagram by superimposing them as described in \cite{Chen-2023} and \cite{chen-hanselman-2026-immersed}. After isotoping the curves to minimize intersections between them, the remaining intersection points will be generators of the $\mathit{gCFK}^-(P_{0,m}(K_{[n]}))$. See also \cite{Levine-nonsurjective} for a different method of computing $\varepsilon$ of satellite knots using bordered Heegaard Floer homology.
\begin{figure}[h!]
    \centering
    \includegraphics[width=0.7\linewidth]{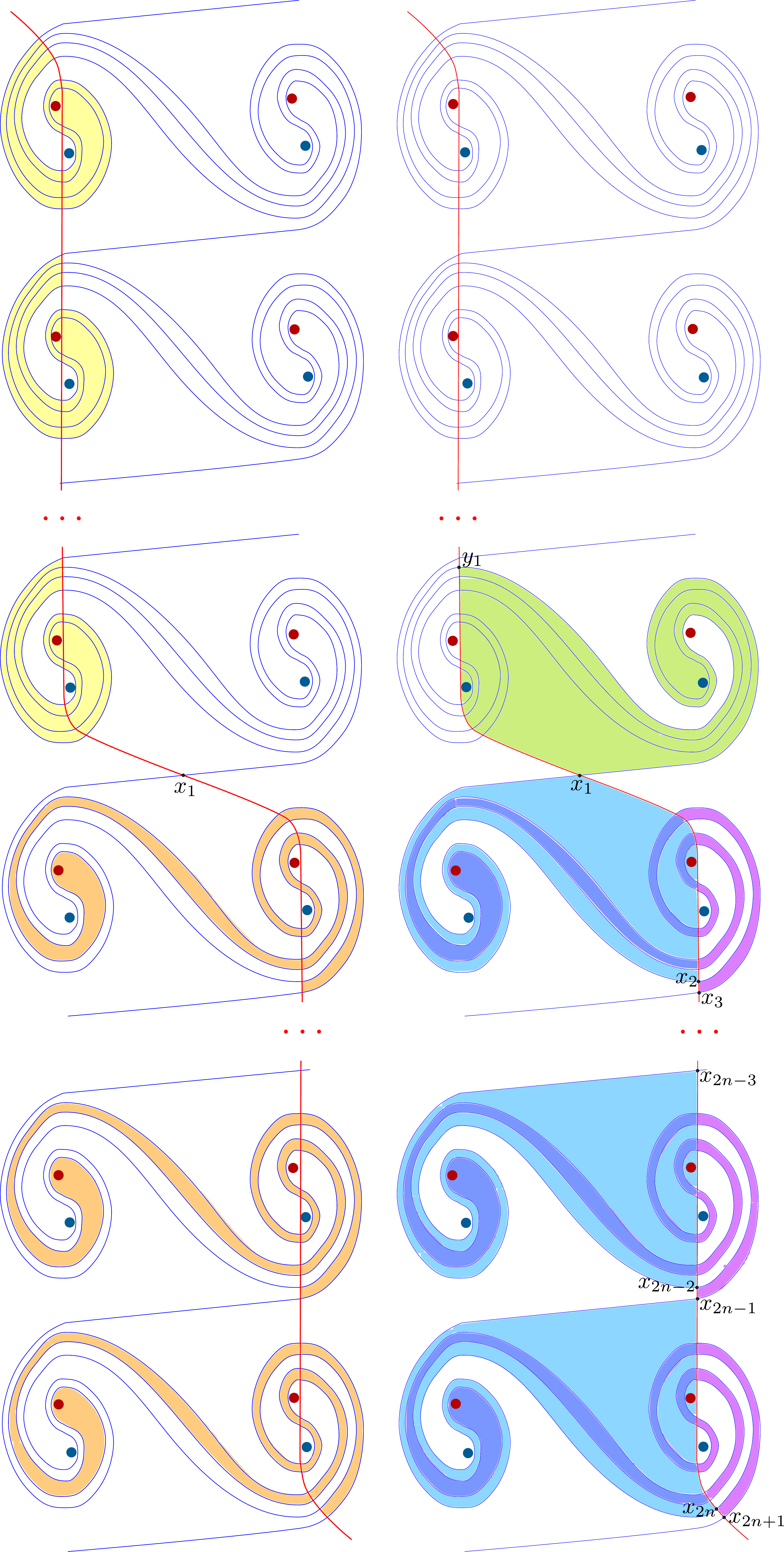}
    \caption{The immersed curves for $P_{0,m}(K_{[n]})$ when $m=1$, $\varepsilon(K)=0$, and $n>0$. On the left are the disks eliminated when eliminating generators with minimal filtration difference. On the right are the disks contributing the relevant differentials in $\mathit{CFK}^-_\mathcal{R}(P_{0,m}(K_{[n]}))$.}
    \label{fig:our case - all disks}
\end{figure}

\begin{figure}[h!]
    \centering
    \includegraphics[width=0.7\linewidth]{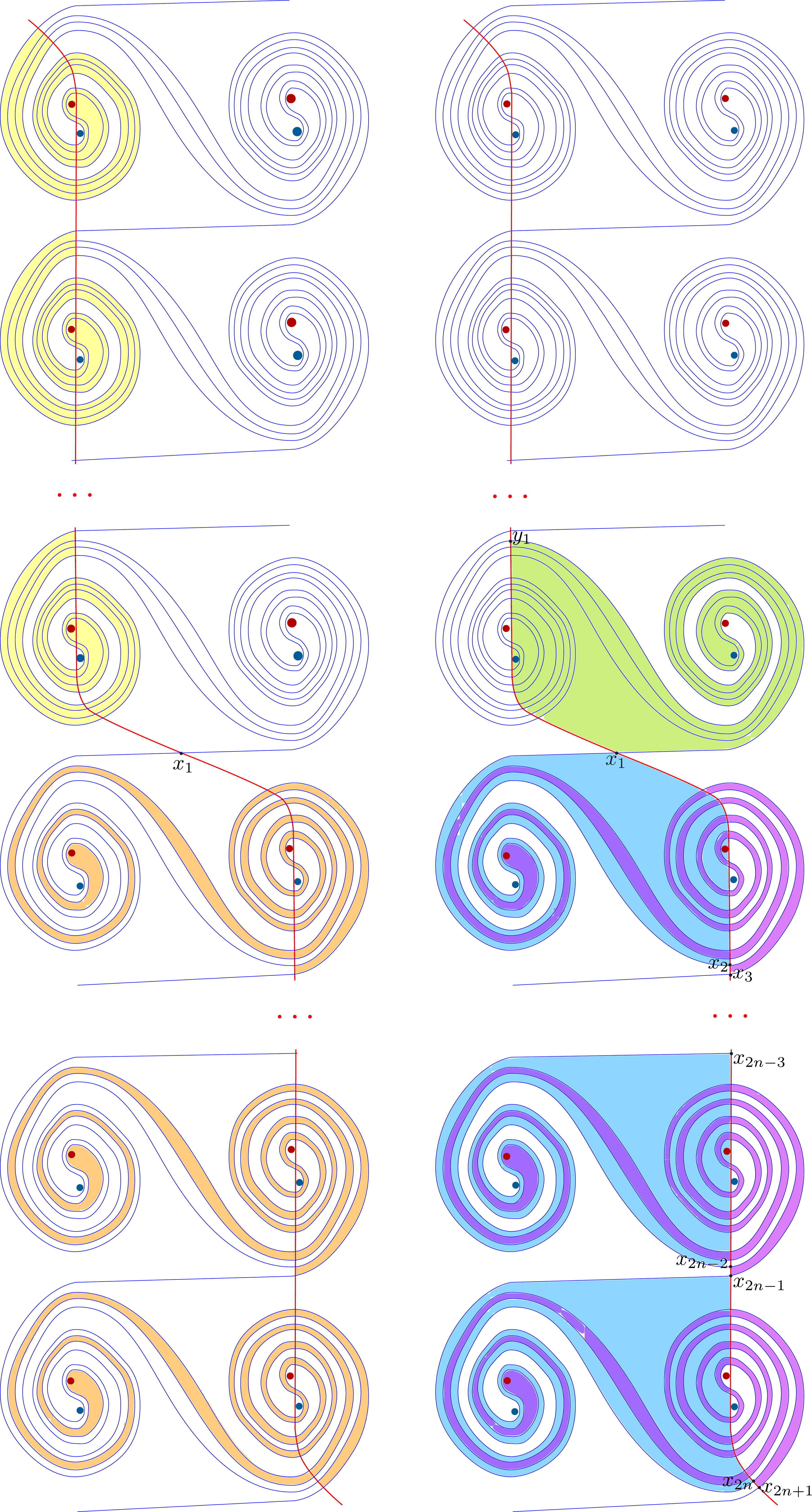}
    \caption{The immersed curves for $P_{0,m}(K)$ when $m=2$, $\varepsilon(K)=0$, and $K$ has framing $n>0$. On the left are the disks eliminated when eliminating generators with minimal filtration difference. On the right are the disks contributing the relevant differentials in $\mathit{CFK}^-_\mathcal{R}$. We observe the same sub-structure as in Figure~\ref{fig:our case - all disks}, where $m=1$.}
    \label{fig:all disks - m - 2}
\end{figure}

\subsection{The case $\varepsilon(K)=0, m\geq 1, n>0$}\label{subsection: case 1}

We begin with the $m\geq 1$ case, point (1) of Theorem~\ref{main theorem}. In this  case, the immersed curves lie as in Figure~\ref{fig:our case - all disks}. We first eliminate disks with minimal filtration difference, which are those shown in yellow and orange on the left side of Figure~\ref{fig:our case - all disks}. We notice that if $m$ increases by $k$, the identified disks will wind around $w$ an additional $k$ times
(see Figure~\ref{fig:all disks - m - 2} for the $m=2$ case), but their elimination will still get rid of all but one intersection point. Thus for any $m\geq 1$, after removing generators with minimal filtration difference, we are left with a single generator, denoted $x_1$ in Figures~\ref{fig:our case - all disks} and \ref{fig:all disks - m - 2} -- specifically, the point invariant under hyperelliptic involution.

To find $\varepsilon$, we consider the chain complex $\mathit{CFK}^-_\mathcal{R}(P_{0,m}(K_{[n]}))$ over the ring $\mathcal{R}=\mathbb{F}_2[U,V]/UV$. We know that after restricting to the horizontal subcomplex, i.e. setting $V=0$, the $\mathbb{F}_2[U]$-summand generator will be part of the subcomplex containing $x_1$; this generator's vertical situation will determine $\varepsilon(P_{0,m}(K_{[n]}))$ (this follows from \cite{chen-hanselman-2026-immersed} and \cite{Chen-2023}; see also \cite[Section 3.2]{Patwardhan-Xiao} for a similar computation). In our case the relevant differentials are as follows (see the right side of Figure~\ref{fig:our case - all disks} for the labeling of the generators and the disks involved):
\[\begin{tikzcd}[cramped,column sep=small]
	{x_2} && {x_1} && {y_1} \\
	{x_4} && {x_3} \\
	\dots && \dots \\
	{x_{2n}} && {x_{2n-1}} \\
	&& {x_{2n+1}}
	\arrow["{U^2}", from=1-1, to=1-3]
	\arrow["U", from=1-1, to=2-3]
	\arrow["{V^2}"', from=1-5, to=1-3]
	\arrow["{U^2}", from=2-1, to=2-3]
	\arrow["U", from=2-1, to=3-3]
	\arrow["U", from=3-1, to=4-3]
	\arrow["{U^2}", from=4-1, to=4-3]
	\arrow["U", from=4-1, to=5-3]
\end{tikzcd}\]
We see that the $\mathbb{F}_2[U]$-summand generator of the horizontal complex is $x_1$, as it can be isolated by a change of basis (specifically adding $U^jx_1$ to generator $x_{2j+1}$ for $j=1, \dots, n$). Thus, the vertical situation of $x_1$ determines $\varepsilon(P_{0,m}(K_{[n]}))$. In our case, we see that there is a $V^2$ disk from $y_1$ into $x_1$; importantly, this disk is unaffected by increasing $m$ via the winding move. Thus, we have an incoming $V$-disk into $x_1$, so 
\[\varepsilon(P_{0,m}(K_{[n]}))=1\]
Note that the same structure of disks is present when we increase $m$ -- it is not affected by the winding move (compare Figures~\ref{fig:our case - all disks} and \ref{fig:all disks - m - 2}). Our conclusion thus holds for all $m\geq 1$. 

\subsection{The case $\varepsilon(K)=0, m\leq -2, n>0$}\label{subsection: case 2}
In this case, the immersed curves lie as in Figure~\ref{fig:immersed curves m neg n pos}. We begin by eliminating the disks with minimal filtration difference shown on the left side of the figure and obtain $x_1$ as the last remaining generator. Again, if we apply the winding move described in Section~\ref{section: beta} to the $\beta$ curve, removing the corresponding disks will still leave exactly one generator, which will occupy the same position as $x_1$. 

\begin{figure}
    \centering
    \includegraphics[width=0.75\linewidth]{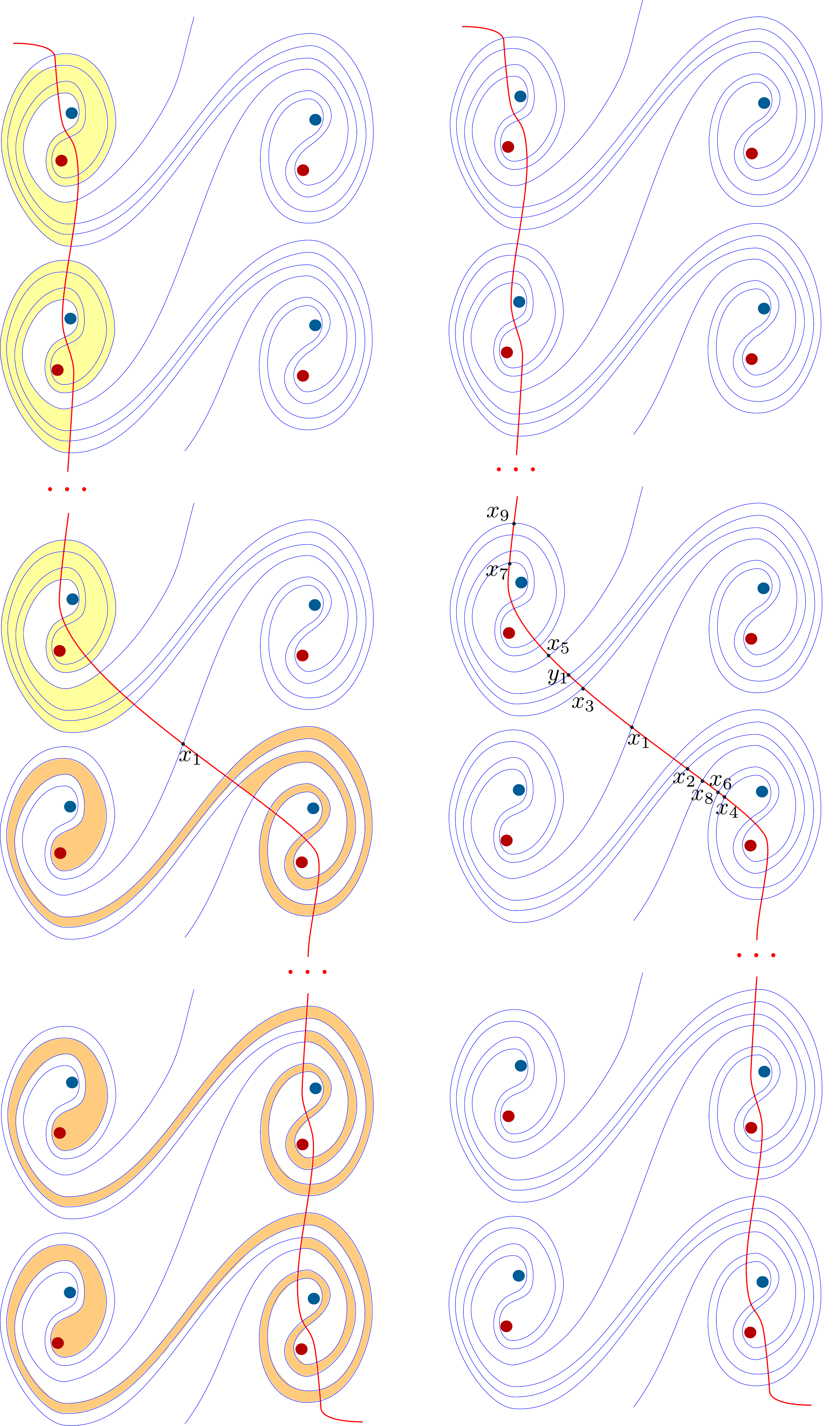}
    \caption{The immersed curves when $\varepsilon(K)=0, m=-2, n>0$. Left: the minimal filtration difference disks isotoped away to get $x_1$ as the remaining generator; right: the generators contributing to relevant differentials of $\mathit{CFK}^-_\mathcal{R}(P_{0,m}(K_{[n]}))$. If we increase $\abs{m}$, i.e. apply the winding move described in Section~\ref{section: beta} to the disks on the right picture, all these disks still exist and form the same sub-structure in $\mathit{CFK}^-_\mathcal{R}(P_{0,m}(K_{[n]}))$.}
    \label{fig:immersed curves m neg n pos}
\end{figure}

To find the value of $\varepsilon$, we consider the subcomplex of $\mathit{CFK}^-_\mathcal{R}(P_{0,m}(K_{[n]}))$ containing $x_1$. The differentials relevant to us are as follows (see the right side of Figure~\ref{fig:immersed curves m neg n pos} for the locations of the generators):
\[\begin{tikzcd}[cramped]
	&& {y_1} &&& \\
	{x_2} && {x_1} & {x_3} & {x_5} \\
	& {x_4} &&& {x_9} & {x_7} \\
	{x_8} && {x_6}
	\arrow["V", from=1-3, to=2-3]
	\arrow["U", from=2-1, to=2-3]
	\arrow["U", from=2-1, to=3-2]
	\arrow["U", from=2-1, to=4-1]
	\arrow["U"', from=2-4, to=2-3]
	\arrow["U", from=2-4, to=2-5]
	\arrow["U", from=2-4, to=3-5]
	\arrow["U", from=2-5, to=3-6]
	\arrow["U", from=3-2, to=4-3]
	\arrow["U", from=3-5, to=3-6]
	\arrow["U", from=4-1, to=4-3]
\end{tikzcd}\]

We see that after an appropriate change of basis (specifically adding $x_1$ to $x_4$ and $x_5$), we obtain that $x_1$ is the $\mathbb{F}_2[U]$-summand generator. Its vertical situation thus determines $\varepsilon$; we see $x_1$ has an incoming $V$-disk from $y_1$, so 
\[\varepsilon(P_{0,m}(K_{[n]}))=1\]

Once again, increasing $\abs{m}$ by applying the winding move described in Section~\ref{section: beta} does not affect this disk structure, so this conclusion holds for all $n > 0, m\leq -2$. 

\subsection{The case $\varepsilon(K)=0, m\leq -2, n<0$} \label{subsection: case 3}
\begin{figure}
    \centering
    \includegraphics[width=0.75\linewidth]{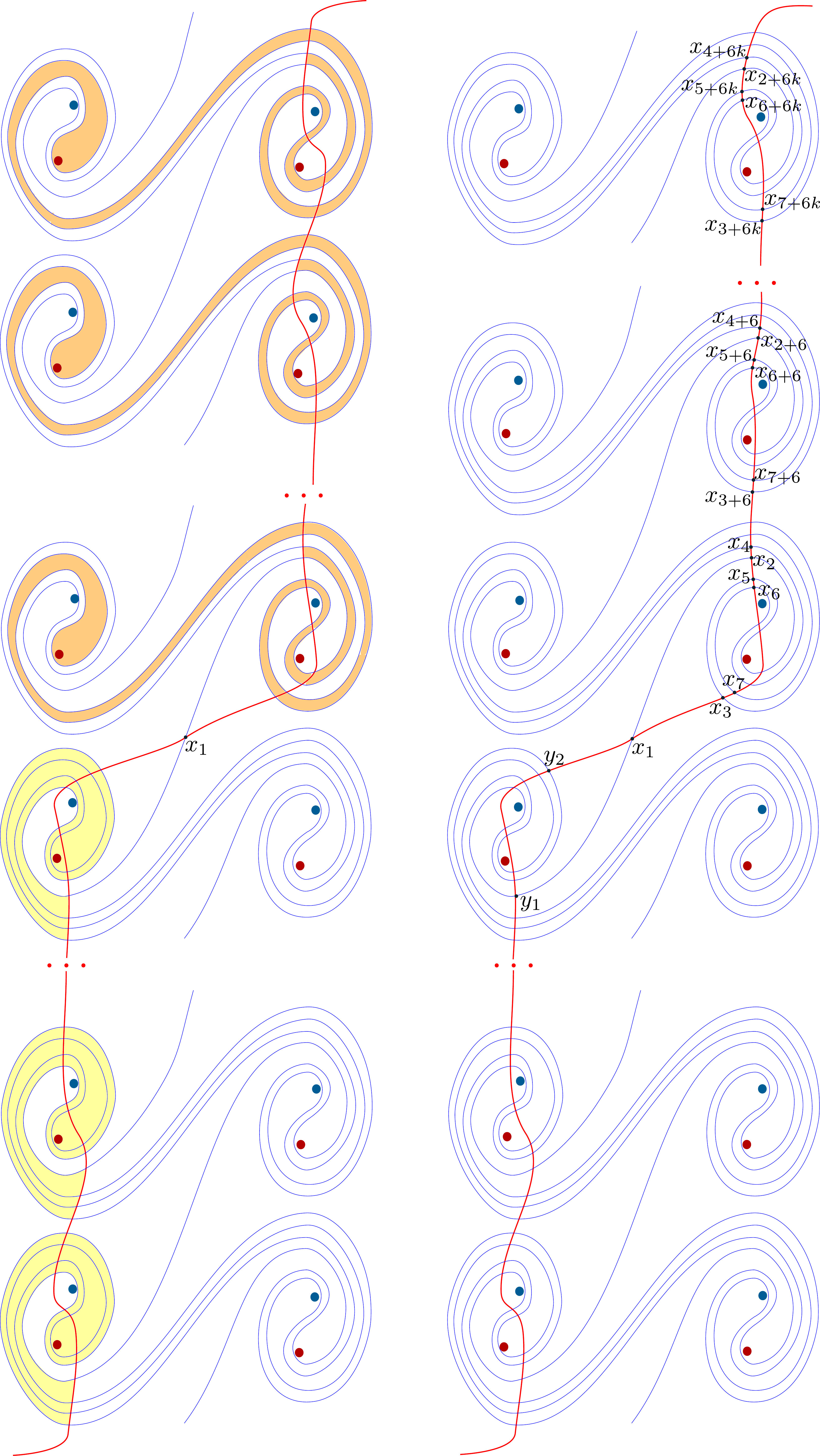}
    \caption{The immersed curves when $\varepsilon(K)=0, m=-2, n<0$. Left: the minimal filtration difference disks isotoped away to get $x_1$ as the remaining generator; right: the generators contributing to relevant differentials of $\mathit{CFK}^-_\mathcal{R}(P_{0,m}(K_{[n]}))$. If we increase $\abs{m}$, i.e. apply the winding move described in Section~\ref{section: beta} to the disks on the right picture, all these disks still exist and form the same sub-structure in $\mathit{CFK}^-_\mathcal{R}(P_{0,m}(K_{[n]}))$.}
    \label{fig:immersed curves m neg n neg}
\end{figure}
In this case, the immersed curves lie as in Figure~\ref{fig:immersed curves m neg n neg}. After isotoping away the disks with minimal filtration shown on the left side of the figure, we are left with the generator $x_1$. Computing the relevant differentials of $\mathit{CFK}^-_\mathcal{R}(P_{0,m}(K_{[n]}))$, we obtain the following structure:
\[\begin{tikzcd}[cramped,column sep=small]
	{y_2} & {y_1} &&&&&&& \\
	{x_3} & {x_1} & {x_2} \\
	{x_7} & {x_5} & {x_4} & {x_{2+6}} & {x_{4+6}} & \dots & {x_{4+6(k-1)}} & {x_{2+6k}} & {x_{4+6k}} \\
	& {x_6} & {x_{3+6}} & {x_{5+6}} & {x_{6+6}} & \dots & {x_{3+6k}} & {x_{5+6k}} & {x_{6+6k}} \\
	&&& {x_{7+6}} &&&& {x_{7+6k}}
	\arrow["U"', from=2-1, to=3-2]
	\arrow["V"', from=2-2, to=1-1]
	\arrow["V"', from=2-2, to=1-2]
	\arrow["U"', from=2-2, to=2-1]
	\arrow["U", from=2-2, to=2-3]
	\arrow["U", from=2-3, to=3-2]
	\arrow["U", from=3-1, to=2-1]
	\arrow["U"', from=3-1, to=4-2]
	\arrow["U"', from=3-3, to=2-3]
	\arrow["{U^2}", from=3-3, to=3-4]
	\arrow["U", from=3-3, to=4-2]
	\arrow["{U^2}", from=3-3, to=4-3]
	\arrow["U", from=3-4, to=4-4]
	\arrow["U"', from=3-5, to=3-4]
	\arrow["{U^2}", from=3-5, to=3-6]
	\arrow["U", from=3-5, to=4-5]
	\arrow["{U^2}", from=3-5, to=4-6]
	\arrow["U"', from=3-7, to=3-6]
	\arrow["{U^2}", from=3-7, to=3-8]
	\arrow["U"', from=3-7, to=4-6]
	\arrow["{U^2}", from=3-7, to=4-7]
	\arrow["U", from=3-8, to=4-8]
	\arrow["U"', from=3-9, to=3-8]
	\arrow["U", from=3-9, to=4-9]
	\arrow["U", from=4-2, to=3-2]
	\arrow["U", from=4-3, to=4-4]
	\arrow["U"', from=4-5, to=4-4]
	\arrow["U", from=4-7, to=4-8]
	\arrow["U"', from=4-9, to=4-8]
	\arrow["U", from=5-4, to=4-3]
	\arrow["U"', from=5-4, to=4-5]
	\arrow["U", from=5-8, to=4-7]
	\arrow["U"', from=5-8, to=4-9]
\end{tikzcd}\]
We see that the $\mathbb{F}_2[U]$-summand generator is $x^*:=x_1+x_4+x_7+\sum_{j=1}^{k} U^j (x_{4+6j}+x_{7+6j})$. Indeed, $\partial_{\text{horz}} x^*=0$, and $x^*$ does not appear in the image of $\partial_{\text{horz}}$ since there are no incoming disks into $x_1$. Moreover, $x_1$ has outgoing $V$-disks into $y_1$ and $y_2$, $x_1$ has no incoming $V$-disks, and no other $x_i$ have $V$-disks into $y_1$ or $y_2$, so we see that $x^*$ will have an outgoing $V$-disk and no incoming $V$-disk. Hence, 
\[\varepsilon(P_{0,m}(K_{[n]})) = -1\]

Again, all the disks described keep the same structure under the winding move described in Section~$\ref{section: beta}$, so our conclusion holds for all $n<0, m\leq -2$.

\subsection{Proof of Theorem~\ref{main theorem}}

The result follows directly from subsections \ref{subsection: case 1}, \ref{subsection: case 2}, and \ref{subsection: case 3}.\qed

\section{Conclusion}\label{section: conclusion}

\subsection*{Proof of Corollary~\ref{main corollary}}
The shaded band in Figure~\ref{fig: pnm qnm} shows $Q_{n,m}$ is concordant to the longitude in $S^1\times D^2$. We thus have that $Q_{n,m}(K)$ is concordant to $K$, so since $\varepsilon$ is a concordance invariant, when $\varepsilon(K)=0$ we obtain
\begin{align*}
    \varepsilon(Q_{n,m}(K))=\varepsilon(K)=0
\end{align*}

On the other hand, in Section~\ref{section: pairing}, we showed that when $n>0$ and $m\geq 1$,
\begin{align*}
    \varepsilon(P_{n,m}(K))=\varepsilon(P_{0,m}(K_{[n]}))=1
\end{align*}
When $n>0$ and $m\leq -2$, we obtained
\begin{align*}
    \varepsilon(P_{n,m}(K))=\varepsilon(P_{0,m}(K_{[n]}))=1
\end{align*}
and when $n<0$ and $m\leq -2$, we obtained 
\begin{align*}
    \varepsilon(P_{n,m}(K))=\varepsilon(P_{0,m}(K_{[n]}))=-1
\end{align*}
By \cite[Theorem 1.2]{homwan2026}, if two knot $n$-traces are diffeomorphic, the absolute values of the $\varepsilon$ invariants of the respective knots must be the same. In all our cases, $\abs{\varepsilon(P_{n,m}(K))}\neq\abs{\varepsilon(Q_{n,m}(K))}$, so we conclude that Corollary~\ref{main corollary} holds.\qed

\subsection*{Further directions} 

A natural next question is whether immersed curves method gives new information about other cases, namely $\varepsilon(K) \neq 0$. 

Using \cite[Table 1]{homwan2026}'s summary of \cite{Bodish-twisted-Mazur}'s $\tau$ and $\varepsilon$ results for $P_{n,0}(K)$ and the relations between the concordance invariants $\tau, \nu, \varepsilon$ (originally in \cite{ozsvath2003knot},  \cite{hom2014bordered}, see \cite[Section 2]{homwan2026} for a convenient list of the properties), we build the following table:

\begin{table}[h!]
    \centering
    \begin{tabular}{c|c|c|c}
         & $\varepsilon(K) = -1$ & $\varepsilon(K) = 1, n<2\tau(K)$ & $\varepsilon(K) = 1, n\geq 2\tau(K)$ \\
         \hline
        $\tau(P_{n,0}(K))$          &  $\tau(K)+1$  &  $\tau(K)+1$  & $\tau(K)$\\
        $\varepsilon(P_{n,0}(K))$   &  $1$          &  $1$          & $1$ \\
        $\nu(P_{n,0}(K))$           &  $\tau(K)+1$  &  $\tau(K)+1$  & $\tau(K)$  \\
         \hline
         $\tau(Q_{n,0}(K))$         &  $\tau(K)$    & $\tau(K)$     & $\tau(K)$\\
         $\varepsilon(Q_{n,0}(K))$  &  $-1 $        & $1$           & $1 $ \\
         $\nu(Q_{n,0}(K))$          &  $\tau(K)+1$  & $\tau(K)$     & $\tau(K)$ \\
    \end{tabular}
    \caption{Values of concordance invariants for $P_{n,0}(K)$ and $Q_{n,0}(K)$.}
    \label{tab:placeholder} 
\end{table}

We build a similar table for the $\varepsilon(K)=0$ cases. The first column was computed in \cite{homwan2026} and the other columns are the cases of Theorem~\ref{main theorem}, for which we use the immersed curves method. Note that by work of \cite[Section 4]{Chen-2023}, the $\tau$ invariant of a satellite knot corresponds to the Alexander grading of the generator we denoted $x_1$, the last remaining generator after isotoping away disks with minimal filtration difference. In all of our cases, $x_1$ is the point invariant under hyperelliptic involution, so its Alexander grading is zero and therefore $\tau(P_{n,m}(K))=0$. We computed the $\varepsilon(P_{n,m}(K))$ invariants above, and can then use properties of $\nu$ to complete the table. We thus obtain the following: 

\begin{table}[h!]
    \centering
    \begin{tabular}{c|c|c|c|c}
         &$m\geq0, n<0$ & $m\leq-2, n<0$& $m\geq 0, n>0$   & $m\leq -2, n>0$ \\
         \hline
         $\tau(P_{n,m}(K))$ &  $1$&  $0$&  $0$&$0$\\
         $\varepsilon(P_{n,m}(K))$&  $1$ &  $-1$ & $1$ &$1$\\
         $\nu(P_{n,m}(K))$& $1$  &  $1$ &$0$ & $0$ \\
         \hline
         $\tau(Q_{n, m}(K))$ &  $0$ & $0$ & $0$&$0$\\
         $\varepsilon(Q_{n, m}(K))$ &  $0$ & $0$ &$0$ &$0$\\
         $\nu(Q_{n, m}(K))$&  $0$ &$0$ &$0$ &$0$\\
    \end{tabular}
    \caption{Values of concordance invariants for $P_{n,m}(K)$ and $Q_{n,m}(K)$ when $\varepsilon(K)=0$, $n\neq0$, $m\neq -1$.}
    \label{tab:epsilon 0} 
\end{table}
We see that when $\varepsilon(K)=0$ and $m\neq -1$, the knot trace invariants $\nu$, $\abs{\varepsilon}$ do not seem to depend on the value of $m$. We posit that this holds more generally:
\begin{conjecture}
    Let $K$ a knot in $S^3$ and $n$ any integer. If $m,m'$ are any integers such that $m \neq -1$ and $m' \neq -1$, then \begin{align*}
        \nu(P_{n,m}(K))=\nu(P_{n,m'}(K)), \quad \abs{\varepsilon(P_{n,m}(K))}=\abs{\varepsilon(P_{n,m'}(K))},\\
        \nu(Q_{n,m}(K))=\nu(Q_{n,m'}(K)), \quad \abs{\varepsilon(Q_{n,m}(K))}=\abs{\varepsilon(Q_{n,m'}(K))}
    \end{align*}
\end{conjecture}

If this conjecture holds, then the $\nu$ and $\abs{\varepsilon}$ values from Table~\ref{tab:placeholder} hold for any $m\neq -1$, rather than only for $m=0$. In that case, these knot concordance invariants will not help distinguish the smooth structures on $X_n(P_{n,m}(K))$ and $X_n(Q_{n,m}(K))$ when $\varepsilon(K)=-1$ or $\varepsilon(K)=1, n\geq 2\tau(K)$. Hom and Wan showed the two $n$-traces form an exotic pair when $\varepsilon(K)=1, n<2\tau(K)$, and $m\geq 0$ (\cite[Theorem 1.3]{homwan2026}). 

The remaining case is $\varepsilon(K)=1, n<2\tau(K), m\leq -2$. This may be achievable with immersed curves, but would require computing $\tau(P_{n,m}(K))$ and $\nu(P_{n,m}(K))$, as the $\abs{\varepsilon}$ invariant seems to be insufficient. 


\clearpage
\bibliographystyle{alpha}
\bibliography{references}
\end{document}

%% file: preamble.tex
\usepackage{graphicx} 
\usepackage{geometry} 
\usepackage{amsmath}
\usepackage{amsthm}
\usepackage{amsfonts}
\usepackage{mdframed}
\usepackage{xcolor}
\usepackage{physics}
\usepackage{tikz}
\usepackage{tikz-cd}
\usepackage[shortlabels]{enumitem}
\usepackage{svg}
\usepackage{mathtools}
\usepackage[makeroom]{cancel}
\usepackage{mathrsfs}

 \usepackage{color}
\usepackage{xcolor}
\definecolor{deepblue}{RGB}{0,0,130}

\usepackage{hyperref}
\hypersetup{
    colorlinks=true,
    citecolor=deepblue,
    linkcolor=deepblue
}

\theoremstyle{definition} 

\newtheorem{theorem}{Theorem}[section]
\newtheorem{conjecture}[theorem]{Conjecture}

\newtheorem{corollary}[theorem]{Corollary}
\newtheorem{remark}[theorem]{Remark}

